\documentclass{article}
\usepackage[a4paper]{geometry}

\usepackage[table]{xcolor}
\definecolor{betterGreen}{rgb}{0, 0.5, 0}

\usepackage{graphicx}

\usepackage{adjustbox}

\usepackage{todonotes}
\usepackage{comment}

\usepackage{authblk}

\usepackage{tikz}
\usetikzlibrary{arrows,calc}

\tikzstyle{arc}=[->, shorten <=3pt, shorten >=3pt, >=stealth, line width=1.5pt]
\tikzstyle{edge}=[shorten <=2pt, shorten >=2pt, >=stealth, line width=1.2pt]
\tikzstyle{redE}=[shorten <=2pt, shorten >=2pt, red, >=stealth, line width=2pt,
                  dash pattern={on 10.5pt off 2pt}]
\tikzstyle{blueE}=[shorten <=2pt, shorten >=2pt, blue, >=stealth,
                   line width=2pt]
\tikzstyle{greenE}=[shorten <=2pt, shorten >=2pt, betterGreen, >=stealth,
                    line width=2pt, dotted]
\tikzstyle{vertex}=[circle, fill=white, draw, minimum size=6pt, inner sep=0pt,
                    outer sep=0pt]
\tikzstyle{block} = [rectangle, rounded corners, minimum width=3cm, text width=5.5cm, minimum height=1cm, text centered, draw=black]

\usepackage{float}

\usepackage{amsmath}
\usepackage{amssymb}

\usepackage{amsthm}

\usepackage{hyperref}
\hypersetup{
  colorlinks
}

\usepackage{enumitem}
\setlist[enumerate]{itemsep=0mm}

\usepackage[capitalize]{cleveref}

\usepackage{multicol}

\newcommand{\IS}{3K_1}
\newcommand{\R}{{}^r\overline{P_3}}
\newcommand{\B}{{}^b\overline{P_3}}
\newcommand{\RR}{{}^rP_3}
\newcommand{\BB}{{}^bP_3}
\newcommand{\RB}{{}^{rb}P_3}
\newcommand{\RT}{{}^rK_3}
\newcommand{\BT}{{}^bK_3}
\newcommand{\RRB}{{}^{rrb}K_3}
\newcommand{\RBB}{{}^{rbb}K_3}

\newcommand{\calF}{\mathcal{F}}
\newcommand{\calP}{\mathcal{P}}
\newcommand{\calH}{\mathcal{H}}
\newcommand{\calJ}{\mathcal{J}}

\DeclareMathOperator{\CSP}{CSP}
\DeclareMathOperator{\NP}{NP}
\DeclareMathOperator{\PO}{P}

\DeclareMathOperator{\minority}{minority}
\DeclareMathOperator{\majority}{majority}

\newtheorem{theorem}{Theorem}
\newtheorem{lemma}[theorem]{Lemma}
\newtheorem{corollary}[theorem]{Corollary}
\newtheorem{proposition}[theorem]{Proposition}
\newtheorem{observation}[theorem]{Observation}
\newtheorem{problem}[theorem]{Problem}
\newtheorem{question}[theorem]{Question}
\newtheorem{remark}[theorem]{Remark}
\newtheorem{claim}{Claim}
\counterwithout{claim}{theorem}

\title{On the expressive power of $2$-edge-colourings of graphs%
\thanks{The authors gratefully acknowledge support from grants UNAM-PAPIIT
IA106425, SEP-CONACYT A1-S-8397 and CONACYT FORDECYT-PRONACES/39570/2020. Nikola
Jedli\v{c}kov\'a was partially supported by SVV–2023–260699. Jan Bok and
Santiago Guzm\'an-Pro were funded by the European Union (ERC, POCOCOP,
101071674). Views and opinions expressed are however those of the author(s) only
and do not necessarily reflect those of the European Union or the European
Research Council Executive Agency. Neither the European Union nor the granting
authority can be held responsible for them.}}

\author[1]{Jan Bok\thanks{jan.bok@matfyz.cuni.cz}} \author[2]{Santiago
Guzm\'an-Pro\thanks{santiago.guzman\_pro@tu-dresden.de}} \author[3]{C\'esar
Hern\'andez-Cruz\thanks{chc@ciencias.unam.mx}}\author[4]{Nikola
Jedli\v{c}kov\'a\thanks{jedlickova@kam.mff.cuni.cz}}
\affil[1]{Department of Algebra, Faculty of Mathematics and Physics, Charles University, Prague, Czech Republic}
\affil[2]{Institut f\"ur Algebra, TU Dresden, Germany}
\affil[3]{Facultad de Ciencias, Universidad Nacional Aut\'onoma de M\'exico, CDMX, M\'exico}
\affil[4]{Department of Applied Mathematics, Faculty of Mathematics and Physics, Charles University, Prague, Czech Republic}

\begin{document}
\date{}

\maketitle
\begin{abstract}
    Given a finite set of $2$-edge-coloured graphs $\calF$ and a hereditary
    property of graphs $\mathcal{P}$, we say that $\calF$ \emph{expresses}
    $\mathcal{P}$ if a graph $G$ has the property $\mathcal{P}$ if and only if
    it admits a $2$-edge-colouring not having any graph in $\calF$ as an induced
    $2$-edge-coloured subgraph. We show that certain classic hereditary classes
    are expressible by some set of $2$-edge-coloured graphs on three vertices.
    We then initiate a systematic study of the following problem. Given  a
    finite set of $2$-edge-coloured graphs $\calF$, structurally characterize
    the hereditary property expressed by $\calF$. In our main results we
    describe all hereditary properties expressed by $\calF$ when $\cal F$
    consists of 2-edge-coloured graphs on three vertices and (1) patterns have
    at most two edges, or (2) $\cal F$ consists of both monochromatic paths and
    a set of coloured triangles.

    On the algorithmic side, we consider the \emph{$\calF$-free colouring
    problem}, i.e., deciding if an input graph admits an $\calF$-free
    $2$-edge-colouring. It follows from our structural characterizations, that
    for all sets considered in (1) and (2) the  $\calF$-free colouring problem
    is solvable in polynomial time. We complement these tractability results
    with a uniform reduction to boolean constraint satisfaction problems which
    yield polynomial-time algorithms that recognize most graph classes
    expressible by a set $\calF$ of $2$-edge-coloured graphs on at most three
    vertices. Finally, we exhibit some sets $\calF$ such that the $\calF$-free
    colouring problem is NP-complete.
\end{abstract}


\section{Introduction}\label{sec:introduction}

It is well-known that every hereditary property is characterized by some set of
forbidden induced subgraphs, often called forbidden \textit{minimal
obstructions}, or forbidden \textit{patterns}.   In many cases, such sets of
forbidden induced subgraphs are infinite, e.g., the class of bipartite graphs is
characterized by the set of all odd cycles. In 1990, Damaschke
\cite{damaschkeTCGT1990} noticed that equipping graphs with an additional
structure results in a more compact set of minimal obstructions for some
hereditary classes.   An \textit{ordered graph} $G_<$, is defined to be a graph
$G$ equipped with a linear order $<$ of its vertex set; the \textit{underlying
graph} of $G_<$ is $G$, and $G_<$ is a \textit{linear ordering} of $G$.   For a
given set $\calF$ of ordered graphs, one can ask, what is the class $\calP$ of
graphs admitting an $\calF$-free linear ordering? A beautiful example found
in~\cite{damaschkeTCGT1990} states that if $\calF$ is the set consisting of
$K_3$ equipped with any linear order of its vertex set, and $P_3$ equipped with
the linear order where the only vertex of degree $2$ is between the two leaves,
then a graph admits an $\calF$-free linear ordering if and only if it is
bipartite. Thus, there are only two ordered minimal obstructions for the class
of bipartite graphs, as opposed to the infinite family of (induced subgraph)
minimal obstructions. Similar examples found in~\cite{damaschkeTCGT1990} include
chordal graphs, forests, and interval graphs. 

Shortly after Damaschke's paper, Duffus, Ginn, and R\"odl~\cite{duffusRSA7}
considered the question above from an algorithmic point of view:  for a given
set $\calF$ of ordered graphs, what is the complexity of deciding if an input
graph $G$ admits an $\calF$-free  linear ordering? Building on this framework,
Hell, Mohar, and Rafiey \cite{hellESA2014} studied sets $\calF$ of linearly
ordered graphs on three vertices and provided a polynomial-time algorithm to
decide if an input graph $G$ admits an $\calF$-free linear ordering.
Equivalently, every graph class defined by means of a forbidden set of linearly
ordered graphs on three vertices can be recognized in polynomial-time. Recently,
Feuilloley and Habib \cite{feuilloleySIDMA35} provided a survey article and
further showed that, with the exception of two classes, all of the
aforementioned classes described by three-vertex patterns can be recognized in
linear time. They also matched these classes with some hereditary properties
easy to describe in well-known terms.

In a different direction,  Skrien~\cite{skrienJGT6} equipped graphs with an
orientation of their edges, and considered the problem of describing a given
family of graphs as the class of graphs admitting an orientation containing no
induced oriented subgraph in a finite set $\calF$ of patterns. He showed that
proper-circular arc graphs, trivially perfect graphs (also known nested interval
graphs),  and comparability graphs are examples of classes for which these
descriptions exist. This work was later extended in~\cite{guzmanArXiv}, where
the authors listed all graph classes defined by forbidden oriented patterns
consisting of exactly three vertices, and in~\cite{guzmanArXiv} they provided
necessary conditions upon certain graph classes for them to admit a
characterization by a finite set of forbidden oriented graphs. In particular, it
follows from their negative results regarding expressibility by finitely many
forbidden equipped graphs, that forests, chordal graphs, and even-hole-free
graphs are not expressible by forbidden orientations. Recently,  from the
algorithmic perspective, Bodirsky and Guzm\'an-Pro \cite{bodirskyArXiv} proved
that for a fixed finite set of finite tournaments $\calF$, there is a complexity
dichotomy (polynomial-time vs.\ NP-complete) for the problems of deciding
whether a given graph admits an $\calF$-free orientation. 

Very recently, \cite{guzmanPHD,guzmanArXiv2} introduced a general framework for
studying equipped graphs, and provided new examples of characterizations arising
from it. Regarding the most recent instances of the study of expressibility by
equipped graphs, we have \cite{guzmanACM438}, where authors equipped graphs with
a circular ordering of its vertex set, and studied classes of graphs that can be
described by forbidding finitely many circularly ordered graphs. It is also
shown that such descriptions can be translated to the linear ordering case.
Similarly, in \cite{paulSTACS2024} Paul and Protopapas extended the linearly
ordered case by considering graphs $G$ equipped with a \textit{tree-layout},
which is a partial ordering of $V(G)$ whose Hasse diagram is a rooted tree. Most
notably, by forbidding certain tree-layouts on three vertices, they defined the
class of proper chordal graphs, which is  a subclass of chordal graphs where the
graph isomorphism problem can be tested in polynomial time.

In this work we consider the problem, analogous to those described above, that
arises when we equip graphs with a (not necessarily proper) $2$-edge-colouring.
This is, given a finite set $\calF$ of $2$-edge-coloured graphs
(\textit{patterns}), we want to know which graphs $G$ admit a $2$-edge-colouring
having no pattern in $\calF$ as an induced $2$-edge-coloured subgraph; in such a
case, we say that the $2$-edge-colouring of $G$ \textit{avoids} $\calF$ or that
it is \textit{$\calF$-free}.   A hereditary property of graphs $\calP$ is
\textit{expressible by forbidden} $2$-edge-coloured graphs when there exists a
finite set of patterns $\calF$ such that a graph $G$ is in $\calP$ if and only
if it admits a $2$-edge-colouring avoiding $\calF$; when it is clear from the
context that we are considering $2$-edge-coloured graphs, we also say that
$\calP$ \textit{is expressed by} $\calF$.  Hence, when we talk about the
\textit{expressive power} of $2$-edge-colouring, we refer to the set of all
hereditary properties expressible by $2$-edge-coloured graphs. The
\textit{$\calF$-free colouring problem} is the problem of deciding whether an
input graph admits an $\calF$-free $2$-edge-colouring.

In this context, the following problems naturally arise:
\vspace{-0.2em}
\begin{enumerate}
  \item Given a finite set of $2$-edge-coloured graphs, find structural
    characterizations of the hereditary property it expresses.

  \item \label{que:two} Given a finite set of $2$-edge-coloured graphs $\calF$,
    determine the complexity of the $\calF$-free colouring problem.
\end{enumerate}

Following a similar program to expression by linear
orderings~\cite{damaschkeTCGT1990,duffusRSA7,feuilloleySIDMA35,hellESA2014} and
by forbidden orientations~\cite{guzmanArXiv,skrienJGT6}, in this work we address
these questions restricted to sets of $2$-edge-coloured graphs on three
vertices.

Our first main result is a characterization of all graph classes expressible by
a set $\calF$ containing both monochromatic paths and some set of
$2$-edge-coloured triangles (Sections~\ref{sec:elementary}
and~\ref{sec:further-elementary}). These include line graphs of bipartite graphs
(\cref{prop:J3J4J8J9}), and line graphs of incidence graphs (\cref{J3J4J7J8J9}).
A nice byproduct of these characterizations is the list of minimal obstructions
to the latter --- the list of minimal obstructions to the former are already
known, and we use them in our proofs. Similarly, our second main result
(\cref{thm:two-edges} in Section~\ref{sec:at-most-two}) characterizes all graph
classes expressible by a set $\calF$ of $2$-edge-coloured graphs on at most $3$
vertices and at most $2$ edges. Some well-structured graph classes expressed by
such sets are semicircular graphs, and co-bipartite graphs. Finally, we leverage
on our structural characterizations to classify the complexity of the
$\calF$-free colouring problem for several sets $\calF$ of $2$-edge coloured
graphs with at most three vertices (Section~\ref{sec:complex}). We also provide
a uniform approach for most of the tractable cases by reducing to tractable
boolean constraint satisfaction problems, namely, to 2-SAT, to Horn-SAT, or to
linear equations over $\mathbb Z_2$. We also exhibit some sets $\calF$ such that
the associated recognition problem is NP-complete.

The rest of the manuscript is structured as follows.  In \cref{sec:prelim}, we
introduce notation and nomenclature that will be used through the rest of the
work.  \cref{sec:simple-fam} includes some basic observations about properties
expressible by $2$-edge-colourings, as well as the first examples of known graph
classes expressible by $2$-edge-colourings. In \cref {sec:elementary}, we study
classes of graphs expressible by a finite set including both monochromatic
copies of $P_3$, obtaining descriptions for most of them in terms of well-known
graph classes\footnote{The family of graphs expressible by the set with exactly
the two monochromatic copies of $P_3$, usually called elementary graphs, arises
naturally when studying claw-free perfect graphs~\cite{chvatalJCTB44}.}.
Building upon these ideas, we deal with classes expressible by forbidding both
monochromatic paths and some set of coloured triangles in
\cref{sec:further-elementary}. In \cref{sec:at-most-two} we present structural
characterizations of all graph classes expressible by a set of $2$-edge coloured
graphs on three vertices and at most $2$ edges. Algorithmic aspects are the main
subject of \cref{sec:complex}; for a variety of finite sets $\calF$ of $2$-edge
coloured graphs with at most three vertices, we classify the hereditary class
expressed by $\calF$ as polynomial-time solvable or NP-complete. The uniform
reduction to boolean CSPs is presented in detail in \cref{sec:uniform}. Two
natural generalizations of the problems we study in this work are presented in
\cref{sec:general}: graphs equipped with a $k$-edge-colouring, and graphs
expressible by $2$-edge-colourings when patterns of order greater than three are
considered. Finally, conclusions and open problems are presented in
\cref{sec:conclusions}.


\section{Preliminaries}
\label{sec:prelim}

For basic terminology and notation, we refer the reader to \cite{bondy2008}.

A \emph{hereditary property} is a class of graphs $\cal P$ such that if $G \in
\cal P$ and $H$ is an induced subgraph of $G$, then $H \in \cal P$. In other
words, it is a class of graphs closed upon taking induced subgraphs. Given a
graph $F$, we say that a graph is \emph{$F$-free} if it does not contain an
induced subgraph isomorphic to $F$. Furthermore, if $\cal{F}$ is a family of
graphs, we say that a graph is \emph{$\cal F$-free} if it is $F$-free for every
$F \in \cal F$.

The \emph{join} of two graphs $G_1$ and $G_2$ is a graph obtained by taking the
disjoint union of $G_1$ and $G_2$ and adding all possible edges with one
endpoint being in $V(G_1)$ and the other one in $V(G_2)$. The \emph{line graph}
of a graph $G$ (denoted by $L (G)$) is the graph where $V(L(G)) = E(G)$ and two
vertices are adjacent if and only if the corresponding edges of $G$ are incident
to a common vertex in $G$.

A graph is \emph{$k$-partite} if its vertex set can be partitioned into $k$
independent sets. A graph is \emph{complete multipartite} if it is $k$-partite
for some $k$ and contains all possible edges between different parts. A graph is
\emph{co-bipartite} if it is the complement of a bipartite graph.
\emph{Clusters} are the graphs obtained as a disjoint union of complete graphs.
They can be also described as $P_3$-free graphs or as the complements of
complete multipartite graphs. We denote by $K_n, K_{i,j}, C_n$, and $P_n$ the
complete graph on $n$ vertices, the complete bipartite graph with parts of size
$i$ and $j$, the cycle on $n$ vertices, and path on $n$ vertices, respectively.

The \emph{Haj\'os graph} is formed by the triangle where for each edge, we
introduce a new vertex and make it adjacent to the endpoints of the respective
edge (see \cref{fig:wonders}). We call $K_{1,3}$ the \emph{claw}. The
\emph{diamond} is the graph obtained from the complete graph $K_4$ by removing
exactly one edge. The \emph{gem} is the join of $P_4$ and $K_1$. The \emph{paw}
is the 4-vertex graph composed of triangle and an extra vertex joined by an edge
with one of the vertices of the triangle. The \emph{$n$-wheel} is the join of
$C_n$ and $K_1$. See \cref{fig:small-graphs} for a depiction of the graphs
defined in the last five sentences. An \emph{odd-hole} is a chordless cycle of
odd length and an \emph{odd anti-hole} is its complement; \emph{even-holes} and
\emph{even anti-holes} are then defined analogously.

\begin{figure}[htb!]
    \centering
    \begin{tikzpicture}
    
        \begin{scope}[scale=0.7,xshift=-2cm]
            \node (L1) at (0,-1.2) {claw};
            \node (c) [vertex] at (0:0){};
            \foreach \i in {0,1,2}{
                \node (\i) [vertex] at ({(360/3)*\i + 90}:1.3){};
                \draw [edge] (c) to (\i);
            }
        \end{scope}
    
        \begin{scope}[scale=0.7,xshift=2cm]
            \node (L1) at (0,-1.3) {paw};
            \node (c) [vertex] at (0:0){};
            \foreach \i in {0,1,2}{
                \node (\i) [vertex] at ({(360/3)*\i + 90}:1.3){};
                \draw [edge] (c) to (\i);
            }
            \draw [edge] (1) to (2);
        \end{scope}
    
        \begin{scope}[scale=0.7,xshift=6cm,yshift=0.3cm]
            \node (L1) at (0,-1.5) {diamond};
            \foreach \i in {0,1,2,3}
                \node (\i) [vertex] at ({(360/4)*\i + 90}:1){};
            \foreach \i in {0,1,2,3}
                \draw [edge] let
                        \n1 = {int(mod(\i+1,4))}
                    in
                        (\i) to (\n1);
            \draw [edge] (0) to (2);
        \end{scope}
    
        \begin{scope}[scale=0.7,yshift=-3.6cm, xshift = -2cm]
            \node (L1) at (0,-1.2) {house};
            \node (t) [vertex] at (90:1.75){};
            \foreach \i in {0,1,2,3}{
                \node (\i) [vertex] at ({(360/4)*\i +45}:1){};
            }
            \foreach \i in {0,1,2,3}
                \draw [edge] let
                        \n1 = {int(mod(\i+1,4))}
                    in
                        (\i) to (\n1);
            \draw [edge] (t) to (0);
            \draw [edge] (t) to (1);
        \end{scope}
    
        \begin{scope}[scale=0.7,xshift=2cm,yshift=-3.3cm]
            \node (L1) at (0,-1.5) {$4$-wheel};
            \node (c) [vertex] at (0:0){};
            \foreach \i in {0,1,2,3}{
                \node (\i) [vertex] at ({(360/4)*\i + 90}:1){};
                \draw [edge] (c) to (\i);
            }
            \foreach \i in {0,1,2,3}
                \draw [edge] let
                        \n1 = {int(mod(\i+1,4))}
                    in
                        (\i) to (\n1);
        \end{scope}
    
        \begin{scope}[scale=0.7,xshift=6cm,yshift=-3.3cm]
            \node (L1) at (0,-1.6) {gem};
            \node (0) [vertex] at (10:1.5){};
            \node (1) [vertex] at (0.7,1){};
            \node (2) [vertex] at (-0.7,1){};
            \node (3) [vertex] at (170:1.5){};
            \node (4) [vertex] at (270:1){};
            \foreach \i in {0,1,2,3}{
                \draw [edge] (4) to (\i);
            }
            \foreach \i in {0,1,2}
                \draw [edge] let
                        \n1 = {\i+1}
                    in
                        (\i) to (\n1);
        \end{scope}

        \begin{scope}[scale=0.7, xshift=10cm, yshift=0.3cm]
            \node (L1) at (0,-1.5) {kite};
            \foreach \i in {0,1,2,3}
                \node (\i) [vertex] at ({(360/4)*\i + 90}:0.9){};
            \foreach \i in {0,1,2,3}
                \draw [edge] let
                        \n1 = {int(mod(\i+1,4))}
                    in
                        (\i) to (\n1);
            \draw [edge] (0) to (2);
           
            \node (n) [vertex] at (0:2){};
            \draw [edge] (3) to (n);
            
        \end{scope}
    
        \begin{scope}[scale=0.7,xshift=10cm,yshift=-3.3cm]
            \node (L1) at (0,-1.6) {butterfly};
            \begin{scope}[xshift=-1cm]
            \foreach \i in {0,1,2}{
                \node (\i) [vertex] at ({(360/3)*\i}:1){};
            }
            \foreach \i in {0,1,2}
                \draw [edge] let
                        \n1 = {int(mod(\i+1,3))}
                    in
                        (\i) to (\n1);
            \end{scope}
            \begin{scope}[xshift=1cm]
            \foreach \i in {0,1,2}{
                \node (\i) [vertex] at ({(360/3)*\i + 180}:1){};
            }
            \foreach \i in {0,1,2}
                \draw [edge] let
                        \n1 = {int(mod(\i+1,3))}
                    in
                        (\i) to (\n1);
            \end{scope}
        \end{scope}
    
\end{tikzpicture}
\caption{Some celebrities in the world of small graphs.}
\label{fig:small-graphs}
\end{figure}
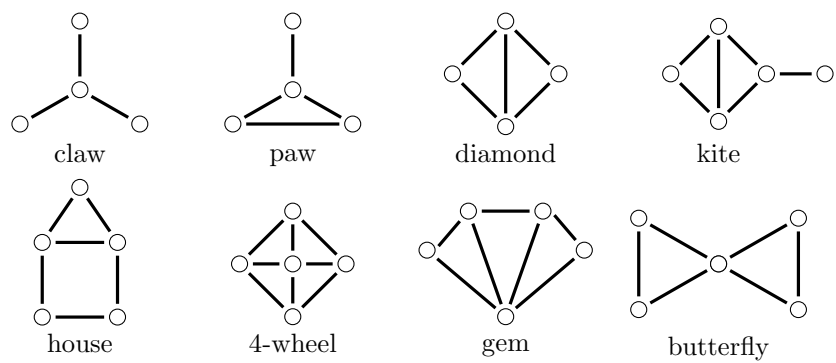

A graph is \emph{semi-circular} if it is an intersection graph of open
semicircles on a circle. A graph $G$ is \emph{perfect} if for every induced
subgraph $H$ of $G$, $\omega(H) = \chi(H)$, where $\omega$ and $\chi$ denote
clique number and chromatic number, respectively. A graph is a
\emph{comparability graph} if it is transitively orientable, i.e.\ its edges can
be directed so that if $(u,v)$ and $(v,w)$ are directed edges, then so is
$(u,w)$. The class of \emph{cographs} has several nice descriptions and
characterizations. One of them says that these are precisely $P_4$-free graphs.
Also, cographs can be defined as graphs which can be constructed from a
single-vertex graph by a repeated use of the following operations: join of two
cographs and disjoint union of two cographs.

We say that two distinct vertices of a graph are \emph{true twins} if they have
the same closed neighbourhoods.


\section{Simple families}
\label{sec:simple-fam}

The family $\calJ$ of all $2$-edge-coloured patterns on three vertices is
depicted in \Cref{fig:3-patterns}; red edges are dashed. We begin our study by
considering subsets of $\calJ$ that express some simple families of graphs.

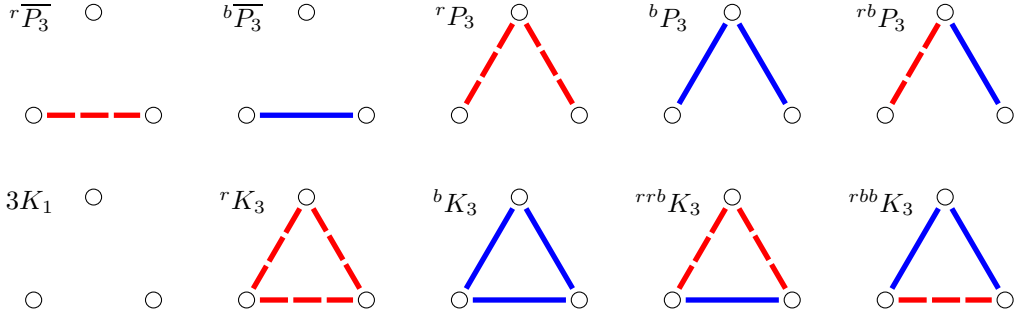
\begin{figure}[htb!]
\centering
\begin{tikzpicture}
\begin{scope}[scale=0.7]  
  
  \begin{scope}
    \node (L2) at (-1.2,1.2) {$\R$}; \foreach \i in {0,1,2} \draw ({(360/3)*\i
    + 90}:1.3) node(\i)[vertex]{};

    \draw [redE] (1) to (2);
  \end{scope}
  
  \begin{scope}[xshift=4cm]
    \node (L2) at (-1.2,1.2) {$\B$}; \foreach \i in {0,1,2} \draw ({(360/3)*\i
    + 90}:1.3) node(\i)[vertex]{};

    \draw [blueE] (1) to (2);
  \end{scope}
  
  \begin{scope}[xshift=8cm]
    \node (L2) at (-1.2,1.2) {$\RR$}; \foreach \i in {0,1,2} \draw ({(360/3)*\i
    + 90}:1.3) node(\i)[vertex]{};

    \draw [redE] (0) to (1); \draw [redE] (2) to (0);
  \end{scope}

  \begin{scope}[xshift=12cm]
    \node (L1) at (-1.2,1.2) {$\BB$}; \foreach \i in {0,1,2} \draw ({(360/3)*\i
    + 90}:1.3) node(\i)[vertex]{};

    \draw [blueE] (0) to (1); \draw [blueE] (2) to (0);
  \end{scope}

    \begin{scope}[xshift=16cm]
    \node (L2) at (-1.2,1.2) {$\RB$}; \foreach \i in {0,1,2} \draw ({(360/3)*\i
    + 90}:1.3) node(\i)[vertex]{};

    \draw [redE]  (0) to (1); \draw [blueE] (2) to (0);
  \end{scope}

  
  \begin{scope}[yshift=-3.5cm]
    \node (L1) at (-1.2,1.2) {$\IS$}; \foreach \i in {0,1,2} \draw ({(360/3)*\i
    + 90}:1.3) node(\i)[vertex]{};
  \end{scope}

  \begin{scope}[xshift=4cm,yshift=-3.5cm]
    \node (L2) at (-1.2,1.2) {$\RT$}; \foreach \i in {0,1,2} \draw ({(360/3)*\i
    + 90}:1.3) node(\i)[vertex]{};

    \foreach \i/\j in {0/1,1/2,2/0} \draw [redE] (\i) to (\j);
  \end{scope}
  
  \begin{scope}[xshift=8cm,yshift=-3.5cm]
    \node (L2) at (-1.2,1.2) {$\BT$}; \foreach \i in {0,1,2} \draw ({(360/3)*\i
    + 90}:1.3) node(\i)[vertex]{};

    \foreach \i/\j in {0/1,1/2,2/0} \draw [blueE] (\i) to (\j);
  \end{scope}

  \begin{scope}[xshift=12cm,yshift=-3.5cm]
    \node (L2) at (-1.2,1.2) {$\RRB$}; \foreach \i in {0,1,2} \draw ({(360/3)*\i
    + 90}:1.3) node(\i)[vertex]{};

      \draw [redE]  (0) to (1); \draw [blueE] (1) to (2); \draw [redE]  (2) to
      (0);
  \end{scope}
  
  \begin{scope}[xshift=16cm,yshift=-3.5cm]
    \node (L2) at (-1.2,1.2) {$\RBB$}; \foreach \i in {0,1,2} \draw ({(360/3)*\i
    + 90}:1.3) node(\i)[vertex]{};
    
    \draw [blueE] (0) to (1); \draw [redE]  (1) to (2); \draw [blueE] (2) to
    (0);
  \end{scope}
\end{scope}
\end{tikzpicture}
\caption{The family $\mathcal{J}$ of all $2$-edge-coloured patterns on $3$
vertices; red edges are dashed.}
\label{fig:3-patterns}
\end{figure}

As a first remark, notice that  if $\calF'$ is obtained from $\calF$ by swapping
the red and blue edges in the $2$-edge-colouring of every graph in $\calF$, then
$\calF'$ expresses the same graph class as $\calF$. Also, notice that the
hereditary class expressed by any singleton subset of $\mathcal{J}$ other than
$\{ \IS \}$ is the family of all graphs; there is always a monochromatic
colouring avoiding a single pattern. The set $\{ \IS \}$ expresses the property
of having independence number at most $2$.   Also, the sets $\{ \R, \B \}$, $\{
\RR, \BB, \RB \}$, and $\{ \RT, \BT, \RRB, \RBB \}$ express the classes of
complete multipartite graphs, clusters, and triangle-free graphs, respectively.
All these observations are particular instances of the following result.

\begin{lemma}
\label{lem:intersection}
  Let $\calF$ be a set of $2$-edge-coloured graphs and $\calH$ be a set of
  graphs. If $\calH'$ is the set of $2$-edge-coloured graphs obtained by
  considering every $2$-edge-colouring of every graph in $\calH$, then
  $\mathcal{F \cup H'}$ expresses the class of $\calH$-free graphs that admit an
  $\calF$-free $2$-edge-colouring.
\end{lemma}
  
\begin{proof}
  It is easy to verify that a graph admits an $\calH'$-free colouring if and
  only if it is $\calH$-free.   Hence, if a graph $G$ admits an $(\calF \cup
  \calH')$-free $2$-edge-colouring, then $G$ is $\calH$-free, and it admits an
  $\calF$-free $2$-edge-colouring.   Conversely, if a graph $G$ is $\calH$-free,
  then any $\calF$-free $2$-edge-colouring of $G$ is also an $(\calF \cup
  \calH')$-free $2$-edge-colouring.
\end{proof}

\begin{remark}\label{rmk:trivial}
    A particular instance of Lemma~\ref{lem:intersection} is when $\calF$ is a
    \emph{trivial set}, i.e., there is a colour $c$ such that for every graph
    $G$, either $\calF$ contains all $2$-edge-colourings of $G$, or $\calF$
    contains no $c$-monochromatic copy of $G$. In this case, $\calF$ expresses
    the class of $\calH$-free graphs, where $\calH$ is the set of graphs $G$
    such that $\calF$ contains all $2$-edge-colourings of $G$.
\end{remark}

The next lemma states a condition for a subset of $\calJ$ to express the class
of all graphs.

\begin{lemma}
\label{lem:trivial-constraint}
  Let $\calF$ be a set of $2$-edge-coloured graphs. There is a graph that does
  not admit an $\calF$-free $2$-edge-colouring if and only if $\calF$ contains
  at least one monochromatic graph of each colour or an empty graph. 
\end{lemma}

\begin{proof}
  If $\calF$ contains neither an empty graph nor a monochromatic red graph, then
  by colouring every edge of a graph red, we obtain an $\calF$-free
  $2$-edge-colouring of that graph.

  For the other direction, if there is an empty graph in $\calF$, then such a
  graph does not admit an $\calF$-free colouring. Otherwise, we can invoke the
  Induced Ramsey Theorem%
  \footnote{According to~\cite{diestel2016,foxAM219}, the Induced Ramsey Theorem
  was discovered independently, around 1973, by Deuber~\cite{deuber1975}, by
  Erd\"os, Hajnal, and P\'osa~\cite{erdos1975}, and by R\"odl~\cite{rodl1973}.}
  (Theorem 9.3.1 in~\cite{diestel2016}), stating that for every graph $H$, there
  exists a graph $G$ such that every 2-edge-colouring of $G$ contains a
  monochromatic induced copy of $H$.
\end{proof}

With similar arguments we can easily characterize the sets $\calF$ that express
a finite class of graphs. Given a positive integer $n$, we denote by ${^b}K_n$
and by ${^r}K_n$ the blue and the red monochromatic clique of order $n$,
respectively.

\begin{lemma}\label{lem:triangles+IS}
    The class of graphs that admit an $\calF$-free colouring is finite if and
    only if there are positive integers $\ell, n,m$ such that $\ell K_1,{^b}K_n,
    {^r}K_m\in \calF$.
\end{lemma}
\begin{proof}
    If $\ell K_1,{^b}K_n, {^r}K_m\in \calF$, it follows from Ramsey's theorem
    that every $2$-edge-colouring of a large enough graph will contain either an
    independent set on $\ell$ vertices, a blue clique on $n$ vertices, or a red
    clique on $m$ vertices. It is straightforward to observe that if $\calF$
    does not contain an independent set, a blue clique, and a red clique, then
    the class of $\calF$-free graphs contains all empty graphs or all complete
    graphs (or both).
\end{proof}

Using Remark~\ref{rmk:trivial}, it is possible to deal with a fair number of
subsets of $\calJ$. Recall that a graph is complete multipartite if and only if
it is $\overline{P_3}$-free.

\begin{proposition}
\label{pro:comp-mult}
  Let $\calF$ be a set of $2$-edge-coloured graphs.   The class of complete
  multipartite graphs is expressed by $\calF$ whenever $\{\R,\B\} \subseteq
  \calF$ and either
  \[
    \calF \subseteq \{\R, \B, \RR, \RB, \RT, \RRB, \RBB\},
  \]
  or, dually,
  \[
    \calF \subseteq \{\R, \B, \BB, \RB, \BT, \RRB, \RBB\}.
  \]
\end{proposition}

Similarly, since $\{3K_1, \overline{P_3}\}$-free graphs are precisely complete
multipartite graphs where each part has at most two vertices, the following
statement is also a direct implication of Remark~\ref{rmk:trivial}.

\begin{proposition}
\label{pro:comp-minus-match}
  Let $\calF$ be a set of $2$-edge-coloured graphs.   The class of complete
  graphs minus a matching is expressed by $\calF$ whenever $\{\IS,\R,\B\}
  \subseteq \calF$ and either
  \[
    \calF \subseteq \{\IS, \R, \B, \RR, \RB, \RT, \RRB, \RBB\},
  \]
  or, dually,
  \[
    \calF \subseteq \{\IS, \R, \B, \BB, \RB, \BT, \RRB, \RBB\}.
  \]
\end{proposition}

We can derive an analogous result for clusters, using similar arguments as
before and the fact that clusters are precisely $P_3$-free graphs.

\begin{proposition}
\label{pro:cluster}
  Let $\calF$ be a set of $2$-edge-coloured graphs.   The class of clusters is
  expressed by $\calF$ whenever $\calF$ contains $\{\RR, \BB, \RB\}$ and either
  \[
    \calF \subseteq \{\R,\RR,\BB,\RB\}
  \]
  or, dually,
  \[
    \calF \subseteq \{\B,\RR,\BB,\RB\}.
  \]
\end{proposition}

To finish with this short stream of results, recall that a graph is a disjoint
union of at most two cliques if and only if it is $\{P_3, 3K_1\}$-free.

\begin{proposition}
\label{pro:two-cliques}
  Let $\calF$ be a set of $2$-edge-coloured graphs.   The class of graphs which
  are a disjoint union of at most two complete graphs is expressed by $\calF$
  whenever $\{\IS,\RR,\BB,\RB\} \subseteq \calF$ and either
  \[
    \calF \subseteq \{\IS,\R,\RR,\BB,\RB\},
  \]
  or, dually,
  \[
    \calF \subseteq \{\IS,\B,\RR,\BB,\RB\}.
  \]
\end{proposition}

We conclude this section with a simple observation.

\begin{observation}
\label{obs:comp-indep}
  The set $\{\R, \B, \RR, \BB, \RB\}$ expresses the class of graphs which are
  either empty or complete.
\end{observation}


\section{Elementary graphs, and line graphs of bipartite graphs.}
\label{sec:elementary}

A graph $G$ that admits a $2$-edge-colouring with no monochromatic induced path
on three vertices (i.e., a $\{\RR,\BB\}$-free colouring) is called an
\textit{elementary graph}, and such a colouring is called an \textit {elementary
colouring} of $G$. These graphs were first considered by Chv\'atal and
Sbihi~\cite{chvatalJCTB44} while studying claw-free perfect graphs, and then
characterized by Maffray and Reed~\cite{maffrayJCTB75} in terms of
\emph{augmentations} of line graphs of bipartite graphs.

In this section, we show that some well-structured subclasses of elementary
graphs are also expressible by forbidden $2$-edge-coloured graphs on at most
three vertices; namely, co-bipartite graphs, line graphs of bipartite graphs,
line graphs of incidence graphs, and graphs with chromatic index at most $2$.

We begin by stating Maffray and Reed's~\cite{maffrayJCTB75} structural
characterization of elementary graphs, for which we introduce the following
definitions. An edge is \textit{flat} if it does not lie in any triangle. Let
$xy$ be a flat edge in a graph $G$, and $B = (X,Y)$ a co-bipartite graph
disjoint from $G$ with at least one $XY$-edge. We construct a graph $G'$ from
$G-\{x,y\}$ and $B$ by adding all edges between vertices in $X$ and vertices in
$N(x)-y$, and all edges between vertices in $Y$ and vertices in $N(y)-x$. In
this case, we say that $G$ is \textit{augmented along} $xy$, that $xy$ is
\textit{augmented}, and that $B$ is the \textit{augment} of $xy$. 

Consider a matching $x_1y_1,\dots, x_ky_k$ of a graph $G$ such that each edge
$x_iy_i$ is a flat edge. For $i\in\{1,\dots, k\}$, let $B_i = (X_i,Y_i)$ be a
co-bipartite graph with at least one $X_iY_i$-edge. We construct a graph $G'$ by
augmenting each edge $x_iy_i$ with augment $B_i$ ($G'$ is the same regardless of
the order in which $G$ is augmented~\cite{maffrayJCTB75}). An
\textit{augmentation} of $G$ is a graph obtained with this construction. 

The five \textit{wonders} are depicted in \cref{fig:wonders}, and we call the
individually: Haj\'os, lighthouse, mausoleum, garden, and colossus.  

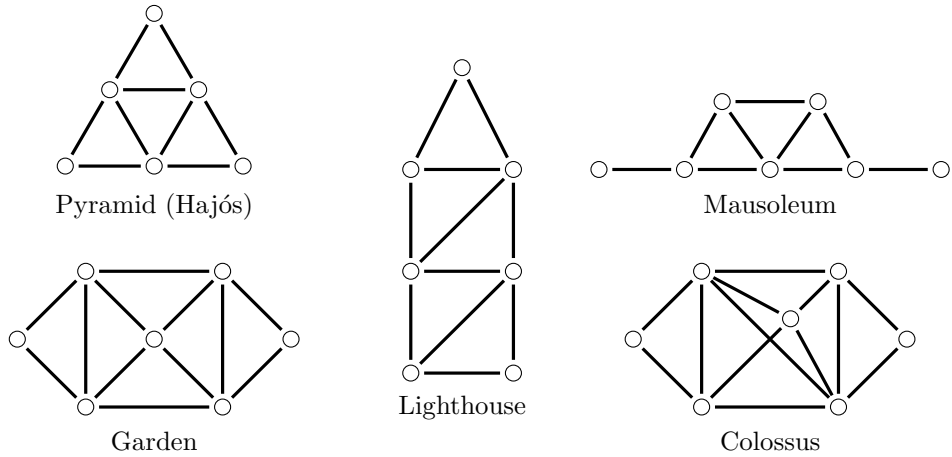
\begin{figure}[ht!]
\centering
\begin{tikzpicture}
\begin{scope}[scale=0.9]
  \begin{scope}[yshift=0.8cm]
    \node (L1) at (0,-1.35) {Pyramid (Haj\'os)};
    \draw (210:1.5) node(0)[vertex]{};
     \draw (270:0.75) node(1)[vertex]{};
      \draw (330:1.5) node(2)[vertex]{};
      \draw (150:0.75) node(3)[vertex]{};
      \draw (30:0.75) node(4)[vertex]{};
      \draw (90:1.5) node(5)[vertex]{};
    
     \foreach \i/\j in {0/1, 1/2, 0/3, 3/1, 1/4, 4/2, 3/4, 3/5, 5/4}
    \draw [edge] (\i) to (\j);
  \end{scope}
  
  \begin{scope}[xshift=4.5cm,yshift=-3cm]
    \node (L2) at (0,-.5) {Lighthouse};
    \draw (-0.75,0) node(0)[vertex]{};
    \draw (0.75,0) node(1)[vertex]{};
    \draw (-0.75,1.5) node(2)[vertex]{};
    \draw (0.75,1.5) node(3)[vertex]{};
    \draw (-0.75,3) node(4)[vertex]{};
    \draw (0.75,3) node(5)[vertex]{};
    \draw (0,4.5) node(6)[vertex]{};
    
    \foreach \i/\j in {0/1, 0/2, 0/3, 1/3, 2/3, 2/4, 2/5, 3/5, 4/5, 6/4, 6/5}
    \draw [edge] (\i) to (\j);
  \end{scope}
  
  \begin{scope}[xshift=9cm]
    \node (L3) at (0,-.5) {Mausoleum};
     \draw (-2.5,0) node(0)[vertex]{};
     \draw (-1.25,0) node(1)[vertex]{};
      \draw (0,0) node(2)[vertex]{};
      \draw (1.25,0) node(3)[vertex]{};
      \draw (2.5,0) node(4)[vertex]{};
      \draw (-.7,1) node(5)[vertex]{};
      \draw (.7,1) node(6)[vertex]{};
    
     \foreach \i/\j in {0/1, 1/2, 2/3, 3/4, 5/6, 5/1, 5/2, 6/2, 6/3}
    \draw [edge] (\i) to (\j);
  \end{scope}

  \begin{scope}[yshift=-3.5cm]
     \node (L1) at (0,-.5) {Garden};
    \draw (-1,0) node(0)[vertex]{};
     \draw (1,0) node(1)[vertex]{};
      \draw (-1,2) node(2)[vertex]{};
      \draw (1,2) node(3)[vertex]{};
      \draw (-2,1) node(4)[vertex]{};
      \draw (0,1) node(5)[vertex]{};
      \draw (2,1) node(6)[vertex]{};
    
     \foreach \i/\j in {0/1, 0/2, 2/3, 3/1, 5/0, 5/1, 5/2, 5/3, 4/0, 4/2, 1/6, 3/6}
    \draw [edge] (\i) to (\j);
  \end{scope}
  
  \begin{scope}[xshift=9cm,yshift=-3.5cm]
   \node (L3) at (0,-.5) {Colossus};
    \draw (-1,0) node(0)[vertex]{};
     \draw (1,0) node(1)[vertex]{};
      \draw (-1,2) node(2)[vertex]{};
      \draw (1,2) node(3)[vertex]{};
      \draw (-2,1) node(4)[vertex]{};
      \draw (.3,1.3) node(5)[vertex]{};
      \draw (2,1) node(6)[vertex]{};
    
     \foreach \i/\j in {0/1, 0/2, 2/3, 3/1, 5/0, 5/1, 5/2, 5/3, 4/0, 4/2, 1/6, 3/6,1/2}
    \draw [edge] (\i) to (\j);
    \end{scope}
\end{scope}
\end{tikzpicture}
\caption{``The five wonders of the non-elementary world''~\cite{maffrayJCTB75}.}
\label{fig:wonders}
\end{figure}

\begin{theorem}\label{thm:elementary-graphs}\cite{maffrayJCTB75}
For a graph $G$, the following statements are equivalent:
\begin{enumerate}
  \item $G$ is an elementary graph.
  \item $G$ is claw-free, perfect, and contains none of the five wonders.
  \item $G$ is an augmentation of the line graph of a bipartite multigraph.
\end{enumerate}
\end{theorem}

Elementary graphs can be equivalently defined in terms of the
\textit{Gallai-graph} construction. Consider a graph $G$. The \textit{Gallai
graph} of $G$ is the graph $Gal(G)$\footnote{This notation was introduced in
\cite{sunJCTB53}, where the author mentions ``This construction was used by
Gallai in his investigation of comparability graphs, hence our notation''.} with
vertex set $E(G)$ where there is an edge $ef$ if $e$ and $f$ induce a $P_3$ in
$G$. It is evident that a $2$-edge-colouring of $G$ without a monochromatic copy
of $P_3$ defines a bipartition of $Gal(G)$. Conversely, every $2$-colouring of
$Gal(G)$ with independent chromatic classes, can be extended to a
$2$-edge-colouring of $G$ with no monochromatic paths on $3$ vertices. Thus, a
graph $G$ is an elementary graph  if and only if $Gal(G)$ is a bipartite graph.
A graph $G$ is \textit{Gallai-perfect}~\cite{sunJCTB53} if $Gal(G)$ is odd-hole
free.  Thus, every elementary graph is Gallai-perfect. 

\begin{proposition}
\label{Gal(G)}
  The following statements are equivalent for a graph $G$:
  \begin{enumerate}
    \item $G$ is an elementary graph.
    \item $Gal(G)$ is a bipartite graph.
    \item $G$ is a claw-free Gallai-perfect graph.
  \end{enumerate}
\end{proposition}

\begin{proof}
The first two items are equivalent as argued in the paragraph preceding this
statement. Also, it is not hard to notice that $G$ contains a claw if and only
if $Gal(G)$ contains a triangle. Thus, $Gal(G)$ is bipartite if and only if it
is odd-hole free and $G$ is claw-free. Therefore, the three statements are
equivalent.
\end{proof}

It immediately follows from \cref{thm:elementary-graphs} that line graphs of
bipartite multigraphs, and co-bipartite graphs are elementary graphs. We show
that co-bipartite graphs also admit  a natural  description by forbidden
$2$-edge-coloured graphs. 

\begin{proposition}
\label{pro:co-bip}
  Let $\calF$ be a set of $2$-edge-coloured graphs. The class of co-bipartite
  graphs is expressed by $\calF$ if $\{\IS, \RR, \BB\}\subseteq \calF$ and
  either
  \[
    \calF\subseteq\{ \IS, \R, \RR, \BB, \RT, \RBB \} \text{ or }
    \calF\subseteq\{ \IS, \R, \RR, \BB, \BT, \RRB \}
  \]
\end{proposition}

\begin{proof}
  We consider the case when $\calF \subseteq \{ \IS, \R, \RR, \BB, \RT, \RBB
  \}$, the remaining one is symmetric. Let $\calF_M =\{ \IS, \R, \RR, \BB, \RT,
  \RBB \}$ and $\calF_0 = \{\IS, \RR, \BB\}$. We will show that any co-bipartite
  graph admits an $\calF_M$-free $2$-edge colouring, and that if $G$ admits a
  $\calF_0$-free $2$-edge-colouring, then $G$ is a co-bipartite graph. Clearly,
  this is sufficient to prove the claim. 

  Let $G$ be a co-bipartite graph. By colouring all edges in each of the two
  cliques blue, and all the remaining edges red, we obtain a $2$-edge-colouring
  of $G$. It is easy to check that this colouring is $\calF_M$-free.

  Now, suppose that $G$ admits an $\calF_0$-free $2$-edge-colouring. In
  particular, $G$ admits a $2$-edge-colouring with no induced monochromatic
  $P_3$, so, by \cref{thm:elementary-graphs}, $G$ is odd-anti-hole-free. Since
  $\IS\in \calF_0$, $G$ is co-triangle-free. Thus, $G$ is a co-bipartite graph.
\end{proof}

\begin{corollary}
\label{cor:J0J3J4}
  The following statements are equivalent for a $3K_1$-free graph $G$:
  \begin{enumerate}
    \item $G$ is a co-bipartite graph.
    \item $G$ is an elementary graph.
    \item $G$ is a Gallai-perfect graph.
  \end{enumerate}
\end{corollary}

\begin{proof}
  The equivalence between the first two items follows from \cref{pro:co-bip},
  and the equivalence between the last two items follows from \cref{Gal(G)} and
  from the fact that $3K_1$-free graphs are claw-free graphs.
\end{proof}

As mentioned in~\cite{sunJCTB53}, the class of Gallai-perfect graphs contains
co-bipartite graphs and bipartite graphs. Contrary to co-bipartite graphs, the
class of bipartite graphs is not contained in the class of elementary graphs
(the claw is not an elementary graph). Nonetheless, the intersection of
bipartite graphs and elementary graphs turns out to have a natural description
by forbidden $2$-edge-coloured graphs.

\begin{proposition}
\label{prop:346789}
  The following statements are equivalent for a graph $G$:
  \begin{enumerate}
    \item $G$ is a bipartite elementary graph.
    \item $G$ has edge-chromatic index at most $2$.
    \item $G$ admits an  $\{\RR,\BB,\RT,\BT,\RRB,\RBB\}$-free
    $2$-edge-colouring.
  \end{enumerate}
\end{proposition}

\begin{proof}
  The first item implies the last one because every bipartite graph $G$ is
  triangle-free, thus any $\{\RR,\BB\}$-free $2$-edge-colouring of $G$ is also
  obviously $\{\RR,\BB,\RT,\BT,\RRB,\RBB\}$-free. Also, since any graph with
  edge-chromatic index at most $2$ is bipartite and admits a $2$-edge-colouring
  with no monochromatic $P_3$, the second item implies the first one. Finally,
  the equivalence between the last two items follows from the definition of
  proper edge-colourings.
\end{proof}

As previously mentioned, line graphs of bipartite multigraphs are elementary
graphs. Now, we show that by considering the restriction to line graphs of
bipartite graphs (without parallel edges), we recover a subclass of elementary
graphs that can be expressed by forbidden $2$-edge-coloured graphs.

\begin{proposition}
\label{prop:J3J4J8J9}
  The following statements are equivalent for a graph $G$:
  \begin{enumerate}
    \item $G$ admits a $\{\RR, \BB, \RRB, \RBB\}$-free $2$-edge-colouring.
    \item $G$ is the line graph of a bipartite graph.
    \item $G$ is a $\{$claw, diamond, odd-hole$\}$-free  graph.
  \end{enumerate}
\end{proposition}

\begin{proof}
  The equivalence between the last two statements was proved in
  \cite{petersonDAM126}. Clearly, no odd-hole nor the claw admit a
  $2$-edge-colouring with no monochromatic $P_3$. It is not hard to observe that
  the diamond does not admit an $\{\RR, \BB, \RRB, \RBB\}$-free
  $2$-edge-colouring. By contrapositive, the previous arguments show that the
  first item implies the last one.

  To conclude the proof we show that if $G$ is the line graph of a bipartite
  graph $H(X,Y)$, then $G$ admits a $\{\RR, \BB, \RRB, \RBB\}$-free
  $2$-edge-colouring. Let $e$ and $f$ be a pair of adjacent vertices in $G$.
  Then, $e$ and $f$ are incident with some common vertex $v$ of $H$. If $v\in
  X$, colour the edge $ef$ blue, and if $x\in Y$, colour $ef$ red. To prove that
  this is a $\{\RR, \BB, \RRB, \RBB\}$-free $2$-edge-colouring of $G$, it
  suffices to show that if $e_1, e_2$, and $e_3$ are three vertices of $G$ such
  that $e_1e_2$ and $e_2e_3$ are edges of $G$, and they are of the same colour,
  then $e_1e_3$ is an edge in $G$ which has the same colour as the other two
  edges, i.e., $e_1,e_2$, and $e_3$ induce a monochromatic triangle. Without
  loss of generality suppose that $e_1e_2$, and $e_2e_3$ are coloured red. By
  the choice of colouring, this means that $e_1$ and $e_2$ are incident with
  some common vertex $y\in Y$, and that $e_2e_3$ are incident with some common
  vertex $y'\in Y$. Notice that if $y\neq y'$, then $e_2$ is incident with two
  different vertices of $Y$, contradicting the fact that $H(X,Y)$ is a bipartite
  graph. Thus, $y = y'$ and so, $e_1$ and $e_3$ are incident with a common
  vertex $y\in Y$. This means that $e_1e_3$ is an edge in $G$ and it is coloured
  red. The claim follows.
\end{proof}

The \textit{incidence graph} of a multigraph $G$ is the bipartite graph $I(G)(E,
V)$ where an edge $e$ is adjacent to a vertex $v$ in $I(G)$ if $e$ is incident
to $v$ in $G$. We say that a bipartite graph $(X,Y)$ is an \textit{incidence
graph} if $(X,Y)$ is the incidence graph of some multigraph. Equivalently, a
graph $G$ is an incidence graph if and only if there is a bipartition $(X,Y)$ of
$G$ such that every vertex in $X$ has degree at most $2$. The butterfly is the
graph consisting of two triangles sharing a single vertex (see
\cref{fig:small-graphs}). For a non-negative integer, the \textit{$n$-butterfly}
is the graph that consists of two triangles joined by a path of length $n$. In
particular, the $0$-butterfly is the butterfly, and the $1$-butterfly is
isomorphic to the graph obtained by adding one edge to $2K_3$. An
\textit{even-butterfly} (resp.\ \textit{odd-butterfly}) is a $2n$-butterfly
(resp.\ $(2n+1)$-butterfly) for $n \ge 0$. 

\begin{theorem}
\label{J3J4J7J8J9}
  The following statements are equivalent for a graph $G$:
  \begin{enumerate}
    \item $G$ admits a $\{\RR, \BB, \BT, \RRB, \RBB\}$-free $2$-edge colouring.
    \item $G$ admits a $\{\RR, \BB, \RT, \RRB, \RBB\}$-free $2$-edge colouring.
    \item $G$ is the line graph of an incidence graph. 
    \item $G$ is a $\{$claw, diamond, odd-hole, even-butterfly$\}$-free.
  \end{enumerate}
\end{theorem}

\begin{proof}
The first two items are trivially equivalent. From \cref{prop:J3J4J8J9}, it
follows that the first item implies that $G$ is a $\{$claw, diamond,
odd-hole$\}$-free. To see that the first item also implies that $G$ is
even-butterfly free, consider a $\{\RR, \BB, \BT, \RRB, \RBB\}$-free $2$-edge
colouring of an $n$-butterfly $B$. Clearly, both triangles of $B$ must be
coloured blue, and since there are no monochromatic blue paths, $n$ must be
greater than $0$ and both end-edges in the path joining the triangles must be
coloured red. Finally, the induced path between both triangles must alternate
colours, and since the end-edges are red, we conclude that $n$ must be odd.
Thus, the first item implies the fourth one.

Now, we see that the third item implies the first one. Let $G$ be the line graph
of an incidence graph $I$, and let $(X,Y)$ be a bipartition of $I$ where every
vertex in $X$ has degree at most $2$. Colour an edge $ef$ of $G$ blue if $e$ and
$y$ are incident with a common vertex $x$ in $X$; otherwise colour $ef$ red. By
replicating the arguments as in the proof of \cref{prop:J3J4J8J9}, one can
notice that this is a $\{\RR, \BB, \RRB, \RBB\}$-free $2$-edge colouring of $G$.
The fact that it also avoids the monochromatic blue triangle, follows from the
assumption that every vertex $x\in X$ has degree at most $2$.

To conclude the proof, we show that the fourth item implies the third one. By
\cref{prop:J3J4J8J9}, we know that $G$ is the line graph of a bipartite graph
$H(X,Y)$. To show that $G$ is the line graph of an incidence graph it suffices
to show that for every connected component $H'$ of $H$, all vertices of degree
at least $3$ in $H'$ are contained in one of the parts of the bipartition.
Without loss of generality, we assume that $H$ is connected, and anticipating a
contradiction, suppose that  there are two vertices $x\in X$ and $y\in Y$ of
degree at least $3$. Choose $x$ and $y$ to minimize $d(x,y)$ among vertices of
degree at least $3$, and consider the shortest $xy$-path $P$, $P =
u_0,u_1,\dots, u_k $ with $u_0 = x$ and $u_k = y$. Since $d(x), d(y)\ge 3$ and
$P$ is a shortest $xy$-path, we can choose two neighbours $x_1, x_2$ of $x$, and
two neighbours $y_1,y_2$ of $y$ such that neither of $x_1,x_2,y_1,y_2$ belong to
$P$.  Let $H'$ be the subgraph of $H$ defined by the edge set $\{xx_1,xx_2,yy_1,
yy_2,xu_1,\dots, u_ky\}$ so $L(H')$ is an induced subgraph of $G$. It is not
hard to see that $L(H')$ consists of  two triangles $xx_1,xx_2,xu_1$ and
$yy_1,yy_2,yu_{k-1}$ joined by a path of length $k-1$. Since $x\in X$ and $y\in
Y$, then $k$ is odd, and thus $L(H')$ is an even-butterfly, contradicting the
choice of $G$. The claim follows.
\end{proof}


\section{Elementary graphs and forbidden coloured triangles.}
\label{sec:further-elementary}

In \cref{sec:elementary} we observed that line graphs of bipartite graphs and of
incidence graphs are expressible by forbidden $2$-edge-coloured graphs.
Moreover, these classes are expressible by forbidding both monochromatic paths
and some set of coloured triangles. In this section we characterize all graph
classes expressible by such a forbidden set. In \cref{fig:small-graphs-2} we
depict further small graphs used in these characterizations, and in
\cref{fig:landscape} we present a landscape of the classes characterized in this
and in the previous section.

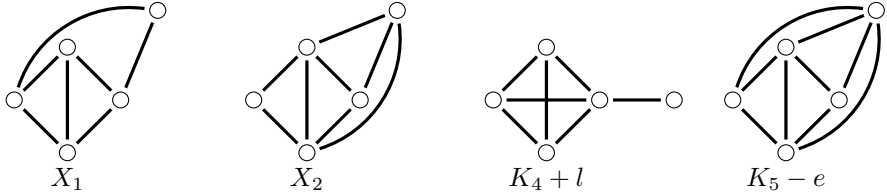
\begin{figure}[ht!]
\centering
\begin{tikzpicture}[scale = 0.7]

    \begin{scope}
                \node (L1) at (0,-1.5) {$X_1$};
        \foreach \i in {0,1,2,3}
            \node (\i) [vertex] at ({(360/4)*\i + 90}:1){};
        \foreach \i in {0,1,2,3}
            \draw [edge] let
                    \n1 = {int(mod(\i+1,4))}
                in
                    (\i) to (\n1);
        \draw [edge] (0) to (2);
        \node (n) [vertex] at (45:2.4){};
        \draw [edge] (3) to (n);
        \draw [edge] (1) to [bend left = 40] (n);
    \end{scope}

    \begin{scope}[xshift = 4.5cm]
        \node (L1) at (0,-1.5) {$X_2$};
        \foreach \i in {0,1,2,3}
            \node (\i) [vertex] at ({(360/4)*\i + 90}:1){};
        \foreach \i in {0,1,2,3}
            \draw [edge] let
                    \n1 = {int(mod(\i+1,4))}
                in
                    (\i) to (\n1);
        \draw [edge] (0) to (2);
        \node (n) [vertex] at (45:2.4){};
        \draw [edge] (2) to [bend right = 40] (n);
        \draw [edge] (0) to (n);
        \draw [edge] (3) to (n);
    \end{scope}

    \begin{scope}[ xshift = 9cm]
        \node (L1) at (0,-1.5) {$K_4+l$};
        \foreach \i in {0,1,2,3}
            \node (\i) [vertex] at ({(360/4)*\i + 90}:1){};
        \foreach \i in {0,1,2,3}
            \draw [edge] let
                    \n1 = {int(mod(\i+1,4))}
                in
                    (\i) to (\n1);
        \draw [edge] (0) to (2);
        \draw [edge] (3) to (1);
        \node (n) [vertex] at (0:2.4){};
        \draw [edge] (3) to (n);
    \end{scope}

    \begin{scope}[xshift = 13.5cm]
        \node (L1) at (0,-1.5) {$K_5-e$};
        \foreach \i in {0,1,2,3}
            \node (\i) [vertex] at ({(360/4)*\i + 90}:1){};
        \foreach \i in {0,1,2,3}
            \draw [edge] let
                    \n1 = {int(mod(\i+1,4))}
                in
                    (\i) to (\n1);
        \draw [edge] (0) to (2);
        \node (n) [vertex] at (45:2.4){};
        \draw [edge] (2) to [bend right = 40] (n);
        \draw [edge] (0) to (n);
        \draw [edge] (3) to (n);
        \draw [edge] (1) to [bend left = 40] (n);
    \end{scope}

\end{tikzpicture}
\caption{Connected graphs on 5 vertices containing the diamond but avoiding the
remaining celebrities from the world of small graphs (see
\cref{fig:small-graphs}).}
\label{fig:small-graphs-2}
\end{figure}

Up to colour symmetry, there are two possible classes expressed by forbidding
three coloured triangles and both monochromatic paths. One is characterized in
Theorem~\ref{J3J4J7J8J9}, we now characterize the other one.

\begin{proposition}\label{prop:BB-RR-BT-RBB}
    The following statements are equivalent for a graph $G$.
    \begin{enumerate}
        \item $G$ admits a $\{\BB,\RR,\BT,\RBB\}$-free $2$-edge-colouring, 
        \item $G$ admits a $\{\BB,\RR,\RT,\RRB\}$-free $2$-edge-colouring, 
        \item each connected component of $G$ is either a line graph of an
            incidence graph or a diamond, and
        \item $G$ is a $\{$claw, $4$-wheel, gem, kite, $X_1,X_2, K_5-e$, 
            odd-hole, even-butterfly$\}$-free graph.
    \end{enumerate}
\end{proposition}
\begin{proof}
    Conditions (1) and (2) are equivalent. It is also easy to observe that (3)
    implies (1) as either $G$ is a line graph of an incidence graph and then
    \cref{J3J4J7J8J9} can be applied. Or $G$ is a diamond, which has a
    $\{\BB,\RR,\BT,\RBB\}$-free $2$-edge-colouring: colour the cycle
    alternatively with red and blue, and finally the chord with red.

    Let us show that (1) implies (4): Suppose $G$ admits a
    $\{\BB,\RR,\BT,\RBB\}$-free $2$-edge-colouring. Then by
    \cref{thm:elementary-graphs}, $G$ does not contain claw and odd-holes. It
    does not contain a $4$-wheel: two of the internal edges of $W_4$ are
    coloured blue in any $\{\BB,\RR\}$-free colouring, implying existence of a
    triangle in the graph with two edges coloured blue. Hence any
    $\{\BB,\RR\}$-free $2$-edge-colouring of $W_4$ is not  $\{\BT,\RBB\}$-free.
    In case of the gem graph, again, avoiding $\BB$ and $\RR$ quickly implies
    that one of the induced triangles has two edges blue. Arguments for
    $\textnormal{kite},X_1, X_2, K_5-e$ are very similar and straightforward,
    let us show it for the kite. In this case, the edges of the diamond incident
    to the vertex adjacent to the one leaf have to be coloured by the same
    colour to avoid $\BB$ and $\RR$ (the opposite colour of the leaf), and
    similarly, the remaining two edges of the (unique) $4$-cycle of the kite
    must also be coloured with the same colour (the same colour as the leaf).
    Hence, any $\{\BB,\RR\}$-free $2$-edge-colouring of the kite induces a
    triangle with two blue edges. Finally, let us also observe that $G$ does not
    contain an induced even-butterfly because any $\{\BB,\RR\}$-free edge
    colouring of an even butterfly forces one end-triangle to have two
    blue-coloured edges, and the other one to have two red-coloured edges.
    Hence, any such $2$-edge-colouring of $G$ contains either $\BT$ or $\RBB$.

    It remains to prove that (4) implies (3): All of the graphs in (4) are
    connected and therefore we can focus on connected graphs only. We
    distinguish two cases. Either $G$ contains a diamond or not. In the former
    case, since claws, and $\textnormal{kite}, X_1, X_2,$ and $K_5-e$ are
    forbidden, it follows that all 5-vertex (connected) graphs containing
    diamonds are forbidden as well, and so $G$ has to be the diamond itself. In
    the latter case, $G$ is a $\{$claw, diamond, odd-hole,
    even-butterfly$\}$-free and hence, by \cref{J3J4J7J8J9}, $G$ is the line
    graph of an incidence graph.
\end{proof}

In \cref{sec:elementary} we observed that line graphs of bipartite graphs and of
incidence graphs are expressible by forbidden $2$-edge-coloured graphs.
Theorem~\ref{thm:elementary-graphs} asserts that augmentations of line graphs of
bipartite graphs are also expressible by forbidden $2$-edge-coloured graphs.
Here, we consider a variation of augmentations of line graphs of bipartite
graphs and of incidence graphs, and we show that these graph classes are again
expressible by forbidden $2$-edge-coloured graphs. The following lemmas build up
to these characterizations.

Let $G$ be the line graph of a bipartite multigraph $H(X,Y)$, and $G'$ an
augmentation of $G$ along a matching of flat edges $e_1f_1,\dots, e_kf_k$. We
say that $G'$ is an \textit{$X$-augmentation} (resp.\ \emph{$Y$-augmentation})
if for each $i\in\{1,\dots,k\}$ there is a vertex $x_i\in X$ (resp.\ $y_i\in Y$)
such that $e_i$ and $f_i$ are incident to $x_i$ (resp.\ to $y_i$). We say that
$G'$ is a \emph{skew-augmentation} if it is either an $X$-augmentation or a
$Y$-augmentation.

Let $G'$ be an augmentation of an elementary graph $G$ along a matching
$x_1y_1,\dots, x_ky_k$. If $(B,R)$ is an elementary colouring of $G$, we say
that $G'$ is a \textit{red-augmentation} (resp.\ \textit{blue-augmentation})
with respect to $(B,R)$ if all edges $x_1y_1,\dots, x_ky_k$ are coloured red
(resp.\ blue).

\begin{lemma}
\label{lem:red-augmentations}
  Let $G$ be an elementary graph with an elementary colouring $(B,R)$. The
  following statements hold for a red-augmentation $G'$ of $G$ with respect to
  $(B,R)$:
  \begin{enumerate}
    \item If $(B,R)$ is a $\RT$-free elementary colouring of $G$, then $G'$
      admits a $\RT$-free elementary colouring.

    \item If $(B,R)$ is a $\RBB$-free elementary colouring of $G$, then $G'$
      admits a  $\RBB$-free elementary colouring. \qed
  \end{enumerate}
\end{lemma}
\begin{proof}
Let $G$ be an elementary graph with an elementary colouring $(B,R)$. Consider an
augmentation $G'$ of $G$ along an edge $xy \in E(G)$ by a co-bipartite graph
$(X,Y)$. Suppose that $xy$ is coloured red, i.e., $xy \in R$. Since $xy$ is a
flat edge, for every $x'\in N(x)$, the vertices $x',x,y$ induce a $P_3$. Thus
$xx'$ is coloured blue and analogously, every edge $yy'$ is coloured blue for
$y'\in N(y)-x$. Now, consider the following extension of $(B,R)$ to the edges of
$G'$: colour blue all edges inside the cliques $X$ and $Y$, all edges of the
form $ux'$ for $u \in X$ and $x'\in N(x)-y$, and all edges $vy'$ for $v \in Y$
and $y' \in N (y)-x$; colour red all $XY$-edges. Clearly, this yields a
$\{\RR,\BB\}$-free $2$-edge-colouring of $G'$. Also, notice that if $xy$ is
coloured red, the colouring defined for $G'$ does not create any new red
triangles. Moreover, if $(B,R)$ is a $\RBB$-free edge-colouring of $G$, then
every edge $uv$ with $u,v\in N(x)-y$  or $u,v\in N(y)-x$ must be coloured blue.
It follows that the previous colouring does not create any copy of $\RBB$. 
\end{proof}

\begin{lemma}
\label{lem:bipartite-multigraphs}
If $G$ is a skew-augmentation of line graph of a bipartite multigraph, then $G$
admits a $\{\RR, \BB, \RBB\}$-free $2$-edge-colouring.
\end{lemma}

\begin{proof}
We first consider the case when $G$ is the line graph of a bipartite multigraph
$H(X,Y)$. Let $e$ and $f$ be a pair of adjacent vertices in $G$. If there is a
vertex $y\in Y$ such that $e$ and $f$ are incident to $y$ (in $H$), then colour
$ef$  blue; otherwise colour $ef$ red. Hence, if a pair of edges $ef, fg$ of $G$
are coloured blue, then there is a pair of vertices $y,y'\in Y$ such that $e$
and $f$ are incident to $y$ (in $H$), and $f$ and $g$ are incident to $y'$ (in
$H$). Since $f$ is incident to $y$ and to $y'$, and $H$ is bipartite, it follows
that $y = y'$, and so, the three edges $e,f,g$  are incident to a common vertex
$y\in Y$. This argument implies that whenever three vertices $e, f, g$ of $G$
induce at least two blue edges, then these vertices induce a monochromatic blue
triangle in $G$. In particular, this colouring of $G$ induces no blue path on
three vertices, nor a copy of $\RBB$. By symmetry, one can notice that whenever
three vertices $e,f,g$ of $G$ induce at least two red edges, these vertices
induce a triangle (but in this case, this triangle might be a monochromatic red
triangle, or a copy of $\RRB$). Thus, for every bipartite multigraph $H(X,Y)$,
there is a $\RBB$-free elementary colouring of its line graph $G$ such that
every edge $gf\in G$ is coloured red whenever $g$ and $f$ are incident to a
common neighbour $x\in X$, and $g$ and $f$ are not parallel edges.

Consider a skew-augmentation $G'$ of $G$ along flat edges $e_1f_1,\dots,e_kf_k$.
Notice that if $ef$ is a flat edge of $G$ (i.e., $ef$ does not belong to any
triangle of $G$), then either $e$ and $f$ are not parallel edges in $H$, or $e$
and $f$ are a pair of isolated parallel edges in $H$ (and in this case, $ef$ is
an isolated edge in $G$). Consider again the previously defined colouring of $G$
with the following modification: colour all isolated edges $ef$ of $G$ red (and
let all the other edges keep their original colour). Since we only modified the
colour of isolated edges, the redefined colouring is still a $\RBB$-free
elementary colouring. Moreover, the following property holds for every flat edge
$ef$ of $G$: either $ef$ is an isolated edge, and thus is coloured red; or there
is a vertex $x\in X$ incident to both $e$ and $f$, and thus also coloured red;
or $e$ and $f$ are incident to a common vertex $y\in Y$ but not to a common
vertex in $X$. Therefore, from the choice of $e_1f_1,\dots, e_kf_k$ and the
definition of skew-augmentation, each edge $e_if_i$ is coloured red. Hence, by
considering the red-augmentation of this redefined colouring, we find a
$\RBB$-free elementary colouring of $G'$ (by part 2 of
\cref{lem:red-augmentations}).
\end{proof}

\begin{lemma}\label{lem:blue-bip} If a graph $G$ admits a $\{\BB,\RR,\BT\}$-free
    $2$-edge-colouring, then $G$ admits a $\{\BB,\RR,\BT\}$-free
    $2$-edge-colouring with no blue odd cycles.
\end{lemma}
\begin{proof}
    Let $R$ be the set of red edges and $B$ the set of blue edges in a
    $\{\BB,\RR,\allowbreak\BT\}$-free $2$-edge-colouring of $G$. Let $c_0,\dots,
    c_n$ be consecutively the vertices of a smallest blue odd cycle $C$ in such
    a colouring (so $n\ge 4$). Since this colouring has no induced monochromatic
    paths and $C$ is a smallest blue odd cycle, for every two vertices of $C$ at
    distance at least 2 from each other, there is an edge in $R$ between them.
    We claim that the colouring of $G$ where the set of red edges is $R' :=
    R\cup E(C)$ and the set of blue edges is $B':=B\setminus E(C)$, i.e., the
    recolouring obtained by changing the colour of the edges of the cycle $C$,
    is a $\{\BB,\RR,\BT\}$-free $2$-edge-colouring of $G$. This colouring is
    clear $\{\BB,\BT\}$-free, and to show that it is $\RR$-free it suffice to
    show that no red edge of the cycle $C$ belongs to a copy of a monochromatic
    path on three vertices. We show that $c_0c_1$ does not belong to such a
    monochromatic path. Let $v\in V(G)$ be a vertex such that $vc_0\in R'$. If
    $v\in V(C)$, then $v,c_0,c_1$ induce a triangle. Now suppose that $v\in
    V(G)\setminus V(C)$, so $vc_0\in R$, i.e., $vc_0$ is a red edge in the
    original colouring of $G$. Also, as noted above, $c_0c_{n-1}\in R$, hence
    there is an edge $vc_{n-1}\in E(G)$. If $vc_{n-1}\in R$, then, using the
    fact that $c_{n-1}c_1\in R$ we conclude that $vc_1\in E(G)$, and so
    $v,c_0,c_1$ do not induce a monochromatic path in the new colouring of $G$.
    Otherwise, $vc_{n-1}\in B$, and in this case the path $v,c_{n-1},c_n$ is a
    blue path on the original colouring of $G$, hence there is an edge $vc_n\in
    E(G)$, and since there are no blue triangles, $vc_n\in R$. It follows from
    the arguments above that $c_nc_1\in R$, so we again conclude that $vc_1$ is
    an edge of $G$ (because $v,c_n,c_1$ is a red path in the original colouring
    of $G$. This proves that there is no vertex $v$ such that $v,c_0,c_1$ induce
    a red path in the new colouring of $G$. Using symmetric arguments we see
    that there is no vertex $v$ such that $v,c_1,c_0$ induce a red path, and so
    $c_1c_0$ does not belong to an induced red path. Therefore, the new
    colouring of $G$ is a $\{\BB,\RR,\BT\}$-free $2$-edge-colouring. Finally,
    after iteratively recolouring all smallest blue odd cycles, we arrive at a
    $\{\BB,\RR,\BT\}$-free $2$-edge-colouring of $G$ with no blue odd cycles.
\end{proof}

\begin{lemma}\label{lem:RR-BB-BT->X-aug} If a graph $G$ admits a $\{\RR, \BB,
\BT\}$-free $2$-edge-colouring, then $G$ is an $X$-augmentation of the line
graph of an incidence graph.
\end{lemma}
\begin{proof}
    Assume without loss of generality that $G$ is connected. Consider a
    $\{\RR,\allowbreak\BB,\allowbreak\BT\}$-free $2$-edge-colouring of $G$ where
    $R$ is the set of red edges, and $B$ the set of blue edges. By
    Lemma~\ref{lem:blue-bip}, we may assume that $(V(G),B)$ is a bipartite
    graph. Consider the equivalence relation $x \sim y$ defined by ``the
    distance between $x$ and $y$ in $(V(G),B)$ is finite and even''. It follows
    by finite induction and the choice of colouring of $G$ that if $x\sim y$,
    then $xy\in R$. Let $H$ be a connected component of $(V(G),B)$, and $(X,Y)$
    a bipartition of its vertex set. Thus, $X$ and $Y$ induce red cliques, and
    if $u\in V(G)\setminus V(H)$ is a neighbour of some $x\in X$, then $ux\in
    R$, and hence $ux'\in R$ for every $x'\in X$ (because the colouring has no
    induced red path on three vertices). It also follows from this observation
    that if $u$ and $v$ are neighbours of some $x$ and $x'$ in $X$, then $uv\in
    R$.  Our goal now is to show that we may assume without loss of generality
    (up to recolouring) that there is no vertex $u\in V(G)\setminus V(H)$ with a
    neighbour in $X$ and a neighbour in $Y$. So suppose that there is such a
    vertex $u\in V(G)\setminus V(H)$ adjacent to some $x\in X$ and some $y\in
    Y$. Hence, for every $x'\in X$ and $y'\in Y$ there is a red path $x',u,y'$
    and so $x'y'\in E(G)$. Putting all these thing together we observe the
    following:
    \begin{itemize}
        \item $V(H)$ induces a clique in $G$ (with some edges being blue and
          some red),
        \item if a vertex $v\in V(G)\setminus V(H)$ is adjacent to some $z\in
          V(H)$, then $vz'\in R$ for every $z'\in V(H)$.
    \end{itemize}
    Hence, for every blue component $H$ with bipartition $(X,Y)$ such that there
    is a common neighbour $u\in V(G)\setminus V(H)$ of some $x\in X$ a and some
    $y\in Y$, we can recolour all blue edges red to obtain a $\{\RR, \BB,
    \BT\}$-free $2$-edge-colouring of $G$. Moreover, this colouring has the
    property that every blue component $H$ has a bipartition $(X,Y)$ such that
    there are no common neighbours of $X$ and $Y$ outside of $H$. We may further
    recolour all the $XY$-edges blue, and not create any monochromatic blue path
    nor a blue triangle. This last modification guarantees that if $x\sim y$,
    $w\sim z$, and $xw,yz\in E(G)$, the $xw$ and $yz$ have the same colour.
    Therefore, this colouring of $G$ defines a canonical $2$-edge-colouring of
    $G/{\sim}$. Moreover, it is straightforward to observe that this is a
    $\{\BB,\RR,\BT,\RBB\}$-free $2$-edge-colouring of $G/{\sim}$, so $G/{\sim}$
    is either a diamond or the line graph of an incidence graph
    (\cref{prop:BB-RR-BT-RBB}). Using the fact that there was no blue component
    $H$ with bipartition  $(X,Y)$ such that $X$ and $Y$ have a common neighbour
    outside $X\cup Y$,  it follows that the blue edges in this colouring of
    $G/{\sim}$ do not belong to a triangle. Therefore, one of the following
    holds:
    \begin{itemize}
        \item $G/{\sim}$ is a diamond, and since every edge of the diamond
          belong to a triangle, $\sim$ is the equality relation on $V(G)$, and
          $G/{\sim} = G$.
        \item Otherwise, $G/{\sim}$ is the line graph of an incidence
          multigraph, no blue edge in the colouring of $G/{\sim}$ belongs to a
          triangle, and by the arguments above, $G$ is an augmentation along
          blue edges of $G/{\sim}$.
    \end{itemize}
    Since the diamond is an augmentation of $K_2$, and $K_2$ is the line graph
    of $P_3$ (which is the incidence graph of $K_2$), we conclude that in either
    of the cases above $G$ is an $X$-augmentation of the line graph of an
    incidence graph.
\end{proof}

In the following statement we talk about \emph{incidence multigraphs}, i.e.,
bipartite multigraphs $H(X,Y)$ where $|N(x)| \le 2$ for every $x\in X$. 

\begin{theorem}
\label{BB-RR-BT-RRB}
    The following statements are equivalent for a graph $G$.
    \begin{enumerate}
        \item $G$ admits a $\{\BB, \RR, \BT\}$-free $2$-edge-colouring.
        \item $G$ admits a $\{\BB,\RR,  \RT\}$-free $2$-edge-colouring.
        \item $G$ admits a $\{\BB,\RR,\BT,\RRB\}$-free $2$-edge-colouring,
        \item $G$ admits a $\{\BB,\RR,\RT,\RBB\}$-free $2$-edge-colouring,
        \item $G$ is an $X$-augmentation of the line graph of an incidence
        multigraph.
        \item $G$ is an $X$-augmentation of the line graph of an incidence
        graph.
    \end{enumerate}
\end{theorem}
\begin{proof}

Again, (1) and (2) are clearly equivalent, and so are (3) and (4). Also, (3) or
(4) imply the first two itemized statements, and (6) implies (5).
Lemma~\ref{lem:RR-BB-BT->X-aug} shows that (1) implies (6).

Let us prove that (5) implies (4). Let $G$ be the line graph of a bipartite
multigraph $H(X,Y)$, and without loss of generality assume that $H(X,Y)$ is
connected. Colour an edge $ef\in E(G)$ blue if $e$ and $f$ are incident to
a common vertex in $Y$ and otherwise red. Notice that all edges $ef\in
E(G)$ where $e$ and $f$ are parallel edges in $H$, are coloured blue. Since
$|N(x)|\le 2$, given any three edges $e_1,e_2,e_3$ incident to a common vertex
$x\in X$,  it must be the case that $e_i$ and $e_j$ are parallel edges from some
$i,j\in\{1,2,3\}$ and $i\neq j$. Hence, there are no blue triangles in this
colouring of $G$. With similar arguments as in the proof of
Lemma~\ref{lem:bipartite-multigraphs} one can see that this is a
$\{\BB,\RR,\RBB\}$-free colouring of $G$. We will now show that every
$X$-augmentation $G'$ of $G$ along flat edges $e_1f_1,\dots, e_kf_k$ admits a
$\{\BB,\RR,\RR,\RBB\}$-free colouring. Notice that in particular, if every flat
edge $e_if_j$ is coloured red, then by \cref{lem:red-augmentations} $G'$ admits
a $\{\BB,\RR,\RR,\RBB\}$-free colouring. Now suppose that there is some flat
edge $ef = e_if_i$ of $G$ coloured blue. This means that $ef$ is a blue edge
that does not belong to a triangle (i.e., is flat) and that $e$ and $f$ are
incident to some common vertex $x\in X$. It follows from the definition of the
colouring of $G$ that $e$ and $f$ are parallel edges in $H$, and let $y\in Y$ be
the other end-vertex of these edges. Since  $e$ and $f$ do not belong to a
triangle in $G$, and $H$ is connected, it must be that $H(X,Y)$ consists of the
two vertices $x,y$ and the two edges $e,f$. Hence, $G$ consists of the single
edge $ef$, and any augmentation of $G$ is a co-bipartite graph. If follows by
Proposition~\ref{pro:co-bip} that $G'$ admits a $\{\BB,\RR,\RR,\RBB\}$-free
edge-colouring.
\end{proof}

We now leverage the ideas of the previous proof to prove our next result.

\begin{theorem}
\label{RR-BB-RRB}
For a graph $G$ the following statements are equivalent:
\begin{enumerate}
  \item $G$ admits a $\{\RR, \BB, \RRB\}$-free $2$-edge-colouring.
  \item $G$ admits a $\{\RR, \BB, \RBB\}$-free $2$-edge-colouring.
  \item $G$ is a skew-augmentation of the line graph of a bipartite multigraph.
\end{enumerate}
\end{theorem}

\begin{proof}
The first two are symmetrically equivalent. We now show that (1) implies (4),
and we proceed with similar arguments as in the proof of ``(1)$\implies$(4)'' in
\cref{BB-RR-BT-RRB}. Let $R$ be the set of red edges and $B$ the set of blue
edges in a $\{\RR, \BB, \RRB\}$-free $2$-edge-colouring of $G$. The graph
$(V(G),R)$ is a disjoint union of (red) cliques $W_1,\dots, W_k$. We claim that
for each blue component $H$ (i.e., a connected component $H$ of $(V(G),B)$), one
of the following holds:
\begin{itemize}
    \item either $H$ is a complete multipartite graph, or
    \item $H$ intersects at most two red cliques.
\end{itemize}
Assume that $H$ intersects at least three cliques. Hence, since $H$ is
connected, there must be at least one vertex $u$ that has a pair of blue
neighbours belonging to different red cliques. Let $U$ be the non-empty set of
vertices with this property. We now claim the following: (a) if $x\in U$ and $y$
is a vertex of $H$ that belongs to a different red clique that $x$, then $xy\in
E(H)\subseteq B$, and (b) $U = V(H)$. Notice that if we prove (a) and (b), it
immediately follows that $H$ is a complete multipartite graph. To prove (a) let
$x\in U$ and $x = x_1,\dots, x_k = y$ be a shortest $xy$-path in $H$.
Anticipating a contradiction, assume that $k\ge 3$. In this case, it follows
with similar arguments as in the proof of \cref{BB-RR-BT-RRB} that $x_i$ belongs
to the same red clique $W_{i_1}$ as $x_1$ if $i$ is odd, and to the same red
clique $W_{i_2}$ as $x_2$ if $i$ is even. Since $y$ belongs to a different red
clique than $x$, then $k$ is even and so $k\ge 4$. By the choice of $x$, there
is some vertex $z$ that belongs to a red clique $W_j$ different from $W_{i_2}$.
Since the $2$-edge-colouring of $G$ has no monochromatic induced paths of length
three, the red edges induce a cluster, and $z$ belongs to a different red clique
than every vertex $x_i$ for $i\in[k]$, it follows by finite induction that $z$
is adjacent to $x_i$ for each $i\in[k]$ and $zx_i \in B$. Hence, $x,z,y$ is a
blue path of length three, and since $y$ and $x$ belong to different red
cliques, then $xy\in  E(H)\subseteq B$. These arguments prove (a), to see that
(b) holds let $x\in U$ and let $y\in V(H)\setminus\{x\}$. If $y$ is in a
different red clique than $x$, then $xy\in E(H)\subseteq B$, and let $z$ be a
neighbour of $x$ belonging to a different red clique than $y$. Using the choice
of colouring of $G$, we conclude that $yz\in E(H)\subseteq B$, and so $y$ has a
pair of neighbours belonging to different red cliques. Now, if $y$ belongs to
the same red clique as $x$, let $w$ be a blue neighbour of $y$ (such a neighbour
exists because $H$ is connected). By (a) we know that $xw\in E(H)\subseteq B$,
and again let $z$ be a blue neighbour of $x$ belonging to a different red clique
than. With similar arguments as before, we see that $zw\in E(H)\subseteq B$, and
in turn this implies that $zy\in E(H)\subseteq B$. Putting all together we
conclude that every component $H$ of $(V(G),B)$ is either a complete
multipartite graph, or it intersects at most two red cliques.

Now, notice that if $H$ is a complete multipartite graph, then every pair of
vertices $x,y$ of $H$ that belongs to the same red clique forms a true twins
pair in $G$. Again, we consider the equivalence relation $x \sim y$ defined by
``$x$ and $y$ belong to  the same red clique, and to a common blue component $H$
that intersects at most two red cliques''. Further consider the equivalence
relation $x\sim_t y$ defined by ``$x$ and $y$ belong to the same red clique, and
to a common blue component $H$ that intersects at least three red cliques''.
Clearly, these equivalence relations are orthogonal, i.e., if $[x]_{\sim_t}$ is
a non-trivial equivalence class, then for every $y$ in this class, the
equivalence class $[y]_{\sim}$ is trivial (and vice versa). Proceeding similarly
to the proof of \cref{lem:RR-BB-BT->X-aug} one can notice that the equivalence
relations $\sim$ and $\sim_t$ respect the edge colouring, and so, the
$2$-edge-colouring of $G$ defines a $2$-edge-colouring of $(G/{\sim_t})/{\sim}$.
It is straightforward to observe that this is a $\{\BB,\RR,\BT,\RBB\}$-free
$2$-edge-colouring of $(G/{\sim_t})/{\sim}$.  
Hence, by \cref{J3J4J7J8J9} $(G/{\sim_t})/{\sim}$ is the line graph of an
incidence graph. Since every pair of vertices equivalent under $\sim_t$ is a
true twin pair,  
$G/{\sim}$ is the line graph of an incidence multigraph. Finally, with similar
arguments as in the proof of \cref{lem:RR-BB-BT->X-aug}, we conclude that $G$ is
an augmentation along blue edges of $G/{\sim}$, i.e., $G$ is a skew-augmentation
of the line graph of a bipartite multigraph.
\end{proof}

Finally, up to colour symmetry there are two remaining cases:
$\{\BB,\RR,\BT,\RT\}$, and $\{\BB,\RR,\BT,\RT, \RBB\}$. We conclude this section
by characterizing the graph classes expressed by these sets.

\begin{proposition}\label{prop:BB-RR-BT-RT} The following statements are
    equivalent for a graph $G$.
    \begin{enumerate}
        \item $G$ admits a $\{\BB,\RR,\BT,\RT\}$-free $2$-edge-colouring, 
        \item each connected component of $G$ is an induced subgraph of either
            $K_5$, the diamond, an even cycle, or a  butterfly, and
        \item $G$ is a $\{$odd-hole, claw, house, bull, gem, $4$-wheel, kite,
            $X_2,X_3, K_4 +l, K_5-e, K_6\}$-free graph.
    \end{enumerate}
\end{proposition}
\begin{proof}
    It is a simple exercise to verify that any connected graph listed in the
    second item admits a $\{\BB,\RR,\BT,\RT\}$-free $2$-edge-colouring. We show
    that the first item implies the third one by contraposition, i.e., no graph
    listed in the third item admits a $\{\BB,\RR,\BT,\RT\}$-free
    $2$-edge-colouring. It follows already from \Cref{thm:elementary-graphs}
    that neither the claw nor $C_5$ admit a $2$-edge-colouring without
    monochromatic paths on three vertices. It is a folklore fact that any
    $2$-edge-colouring of $K_6$ contains a monochromatic triangle. Now, notice
    that if $G$ is the house, then any $2$-edge-colouring of $G$ avoiding
    induced monochromatic copies of $P_3$ forces that the unique triangle is
    coloured monochromatically, hence the house does not admit a
    $\{\BB,\RR,\BT,\RT\}$-free $2$-edge-colouring. With similar arguments one
    can notice that neither the bull, the gem, nor the $4$-wheel admit a
    $\{\BB,\RR,\BT,\RT\}$-free $2$-edge-colouring. It is also a simple exercise
    noticing that any $\{\BB,\RR\}$-free colouring of $K_4+l, \textnormal{kite},
    X_2,X_3$ or of $K_5-e$ contain a monochromatic triangle. We illustrate an
    argument for the latter. Consider a colouring of $K_5-e$ with no
    monocrhomatic $P_3$, and let $u_1,u_2$ be the pair of non-adjacent vertices
    of $K_5-e$. So, for any $i\in\{3,4,5\}$ the vertices $u_1,u_i,u_2$ induce a
    path on three vertices, and  these paths must be coloured
    heterochromatically. Also, by the pigeonhole principle, there are two
    $i,j\in\{3,4,5\}$ with $i\neq j$ such that the colours of the edges $u_iu_1$
    and $u_ju_1$ coincide --- without loss of generality we assume these edges
    are coloured red. Hence, $u_iu_2$ and $u_ju_2$ are coloured blue, and hence,
    regardless of the colour of the edge $u_iu_j$ we find a monochromatic
    triangle in such a colouring of $K_5-e$.
    
    Finally, we show that the third itemized statement implies the first one.
    Since the forbidden graphs listed in the third itemized statement are
    connected, it suffices to prove the second statement for a connected graph
    $G$. First notice that if $G$ contains four vertices $v_1,v_2,v_3,v_4$
    inducing a diamond, then $V(G) = \{v_1,v_2,v_3,v_4\}$. Indeed, if $V(G)
    \setminus \{v_1,v_2,v_3,v_4\}$ is not empty, and since $G$ is connected,
    then a vertex $v_5$ adjacent to at least one of $v_1,v_2,v_3,v_4$ exists.
    Any possible configuration of adjacencies between $v_5$ and
    $v_1,v_2,v_3,v_4$ create one of the forbidden graphs in $\{$claw, $4$-wheel,
    kite, $X_2,X_3,K_5-e\}$, contradicting the choice of $G$. Now, we assume
    that $G$ is also diamond-free. With similar arguments as before, one can
    notice that if $G$ contains a clique on five vertices, then $G\cong K_5$,
    and if $G$ is $K_5$-free and contains a clique on four vertices, then
    $G\cong K_4$.  Assuming that $G$ is also $\{K_4$, diamond$\}$-free, the
    reader may use a similar technique to notice that if $G$ contains an even
    cycle, then $G$ is an even cycle. To conclude the proof we show that if $G$
    avoids the forbidden graphs in the last itemized statement and $G$ is also
    $\{K_4,$ diamond, even-hole$\}$-free, then $G$ is an induced graph of a
    butterfly. Most of the graph listed in the fourth contain either a diamond
    or a complete graph on four vertices, so it suffices to assume that $G$ is a
    $\{$claw, hole, diamond, house, bull, $K_4\}$-free graph. 
\end{proof}

\begin{proposition}\label{prop:BB-RR-BT-RT-RBB}
    The following statements are equivalent for a graph $G$.
    \begin{enumerate}
        \item $G$ admits a $\{\BB,\RR,\BT,\RT, \RBB\}$-free $2$-edge-colouring,
        \item $G$ admits a $\{\BB,\RR,\BT,\RT, \RRB\}$-free $2$-edge-colouring,
        \item each connected component of $G$ is an induced subgraph of either
          $K_4$, the diamond, an even cycle, or an odd butterfly, and
        \item $G$ is a $\{K_5,$ odd-hole, claw, house, bull, $K_4+l$, kite,
          $X_2, X_3$, $K_5-e$, gem, $4$-wheel, even-butterfly$\}$-free graph.
    \end{enumerate}
\end{proposition}

\begin{proof}
    Items (1) and (2) are clearly equivalent. The fact that (4) implies (3)
   follows from Proposition~\ref{prop:BB-RR-BT-RT}, and it is straightforward to
   observe that any graph listed in (3) admits a $\{\BB,\RR,\BT,\RT,
   \RBB\}$-free $2$-edge-colouring.
    
    Finally, we show that if $G$ admits a $\{\BB,\RR,\BT,\RT, \RBB\}$-free
    $2$-edge-colouring, then $G$ avoids the forbidden graphs from (4):
    \begin{itemize}
        \item $G$ is $K_5$-free because there is a unique (up to isomorphism)
          $2$-edge-colouring of $K_5$ avoiding both monochromatic triangles;
          namely each colour class induces a $5$-cycle. It is straightforward to
          observe that this colouring contains a copy of $\RBB$.
        \item $G$ is $\{C_5,$ claw, house, bull, gem, $4$-wheel, $K_4+l,
          \textnormal{kite}, X_2, X_3, K_5-e\}$-free by
          Proposition~\ref{prop:BB-RR-BT-RT}.
        \item $G$ is even-butterfly-free by Proposition~\ref{prop:BB-RR-BT-RBB}
          (because in particular, $G$ admits a
          $\{\BB,\RR,\allowbreak\BT,\allowbreak\RBB\}$-free colouring). \qedhere
    \end{itemize}
\end{proof}

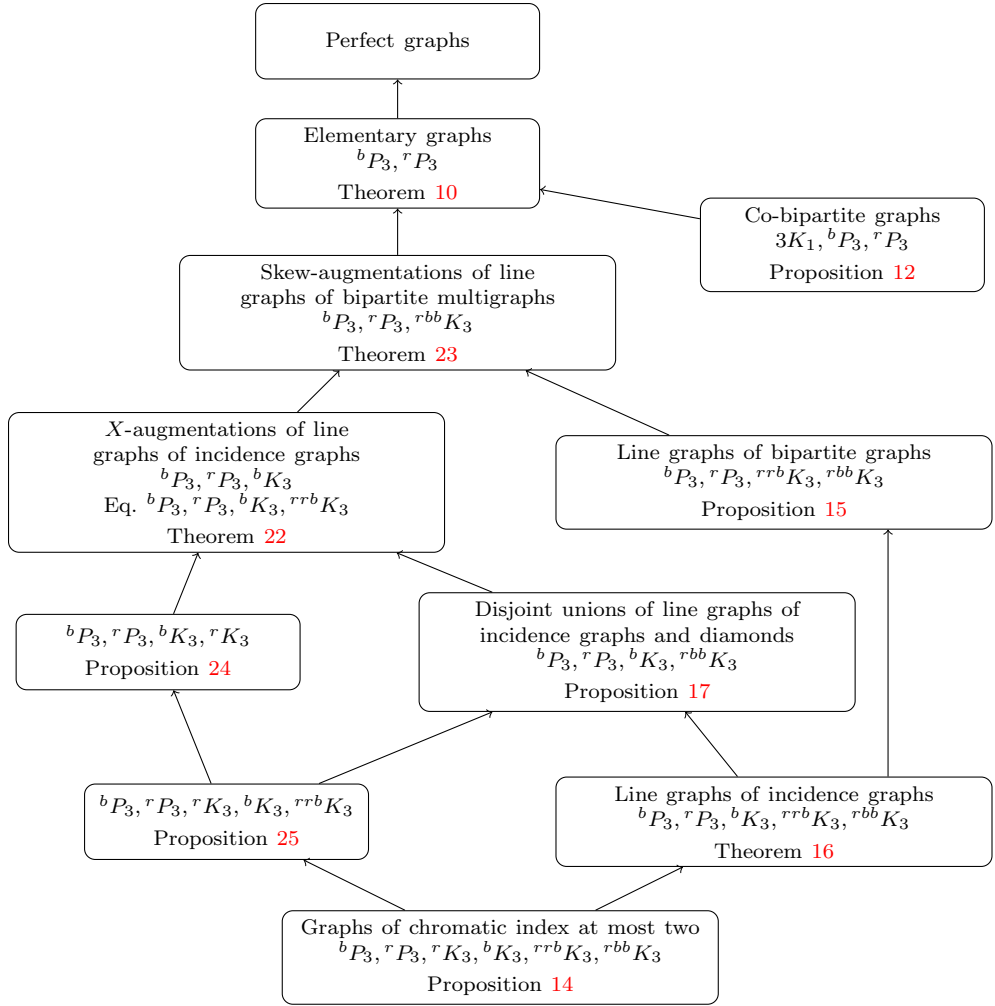
\begin{figure}[htb!]
\centering
\begin{tikzpicture}[scale = 0.9]
    
    \node[block, minimum width=3cm] (allT) at (0,0.5) {\footnotesize Graphs of chromatic index at most two \\$\BB,\RR,\RT, \BT, \RRB, \RBB$\\ \cref{prop:346789}};

    \node[block] (BT-RRB-RBB) at (4,2.5) {\footnotesize  Line graphs of incidence graphs
    \\$\BB,\RR,\BT, \RRB, \RBB$\\ \cref{J3J4J7J8J9}};
    \node[block, minimum width=1cm, text width = 3.5cm] (BT-RT-RBB) at (-4,2.5) {\footnotesize $\BB,\RR,\RT, \BT, \RRB$\\ \cref{prop:BB-RR-BT-RT-RBB}};

    \node[block] (RRB-RBB) at (4,7.5) {\footnotesize Line graphs of bipartite graphs \\$\BB,\RR,\RRB, \RBB$\\ \cref{prop:J3J4J8J9}};
    \node[block] (BT-RBB) at (2,5) {\footnotesize Disjoint unions of line graphs of incidence graphs and diamonds \\$\BB,\RR,\BT, \RBB$\\ \cref{prop:BB-RR-BT-RBB}};
    \node[block, minimum width=1cm,  text width = 3.5cm] (BT-RT) at (-5,5) {\footnotesize $\BB,\RR,\BT, \RT$\\ Proposition~\ref{prop:BB-RR-BT-RT}};

    \node[block, minimum width=1cm,  text width = 3.5cm] (co-bip) at (5,11) {\footnotesize Co-bipartite graphs \\$\IS, \BB,\RR$ \\ \cref{pro:co-bip}};

    \node[block] (RBB) at (-1.5,10) {\footnotesize Skew-augmentations of line graphs of bipartite multigraphs
    \\$\BB,\RR,\RBB$\\ \cref{RR-BB-RRB}};
    
    \node[block, text width = 3.5cm] (EL) at (-1.5,12.2) {\footnotesize Elementary graphs
    \\$\BB,\RR$ \\ \cref{thm:elementary-graphs}};

    \node[block, text width = 3.5cm] (PER) at (-1.5,14) {\footnotesize Perfect graphs};
    
    \node[block] (BT) at (-4,7.5) {\footnotesize $X$-augmentations of line graphs of incidence graphs \\$\BB,\RR,\BT$ \\Eq.\ $\BB,\RR,\BT,\RRB$\\ \cref{BB-RR-BT-RRB}};

    \foreach \from/\to in {EL/PER, RBB/EL, allT/BT-RRB-RBB, allT/BT-RT-RBB, BT-RT-RBB/BT-RBB, BT-RRB-RBB/BT-RBB,
    BT-RT-RBB/BT-RT, BT-RT/BT, BT-RBB/BT, RRB-RBB/RBB, BT/RBB, co-bip/EL}     
    \draw [->] (\from) to (\to);

     \draw [->, transform canvas={xshift=1.5cm}] (BT-RRB-RBB) to (RRB-RBB);
    
\end{tikzpicture}
\caption{The landscape of graph classes characterized in
Sections~\ref{sec:elementary} and~\ref{sec:further-elementary}. In particular, 
of all graph classes expressible by a set 
$\calF$ containing both monochromatic paths, and
$\calF\subseteq\{\BB,\RR,\BT,\RT,\RRB,\RBB\}$.}
\label{fig:landscape}
\end{figure}


\section{Patterns with at most two edges}
\label{sec:at-most-two}

In this section, we answer question \ref{que:two} from the introduction for sets
$\calF$ of 2-edge-coloured graphs  on at most two edges and three
vertices\footnote{Notice that question 3 has a simple solution for such sets
$\calF$: deciding if an input  graph $G$ admits an $\calF$-free orientation can
be solved in polynomial time via a straightforward reduction to 2-SAT. Such
reductions generalize to arbitrary finite sets of forbidden 2-edge-coloured
graphs; we discuss them in \cref{sec:complex}.}, i.e., we propose structural
characterizations of each graph classes expressible by such a set $\calF$. We
begin with a series of results that build up to the main theorem of this section
(\cref{thm:two-edges}).

\begin{proposition}
\label{pro:23}
  For a graph $G$, the following statements are equivalent:
  \begin{enumerate}
    \item $G$ admits a $\{\B,\RR\}$-free colouring.
    \item $G$ admits a $\{\R,\BB\}$-free colouring.
    \item $G$ is a $\{K_1+P_3,C_5\}$-free graph.
    \item $G$ is a join of clusters and $\{3K_1, C_5\}$-free graphs.
  \end{enumerate}
\end{proposition}

\begin{proof}
  The equivalence between the first two statements is trivial. For the
  equivalence between the second and third statement, it is evident that neither
  $K_1+P_3$ nor $C_5$ admit a $\{\B,\RR\}$-free colouring. Now, we show that any
  $\{K_1+P_3,C_5\}$-free graph admits a $\{\B,\RR\}$-free colouring. Let $G$ be
  such a graph. If an edge $e$ belongs to some induced copy of $\overline{P_3}$,
  colour it red; otherwise, colour $e$ blue. Clearly, this colouring is
  $\B$-free. If $u,v,w$ forms an induced copy of $P_3$ in $G$, then both its
  edges are coloured red if and only if there are vertices $x$ and $y$ such that
  $xw, xv, yv$ and $yu$ are not edges of $G$.   If $xu$ or $yw$ are not edges of
  $G$, then $G$ contains an induced copy of $K_1 + P_3$.  Therefore, $xu$ and
  $yw$ are edges of $G$. But this is impossible, since $\{ u, v, w, x, y \}$
  would induce either $K_1+P_3$ or $C_5$, depending on whether $xy$ is an edge
  of $G$.

  To prove the equivalence between the last two statements, notice that
  $K_1+P_3$ is the complement of the paw. Olariu \cite{olariuIPL28} showed that
  a graph $H$ is paw-free if and only if each component of $H$ is triangle-free
  or complete multipartite. Since $C_5$ is a self-complementary graph, if $G$ is
  a $\{K_1+P_3,C_5\}$-free graph, then it is a join of $\{3K_1,C_5\}$-free
  graphs and clusters, since all clusters are $C_5$-free. On the other hand,
  neither $K_1+P_3$ nor $C_5$ are a join of clusters and $\{3K_1, C_5\}$-free
  graphs. This completes the last equivalence.
\end{proof}

\begin{corollary}
\label{cor:023}
  The class of $\{3K_1, C_5\}$-free graphs is expressed by any of the sets
  $\{\IS, \B, \RR\}$ or $\{\IS, \R, \BB\}$.
\end{corollary} 

\begin{proposition}
\label{pro:234}
  For a graph $G$,  the following statements are equivalent:
  \begin{enumerate}
    \item $G$ admits a $\{\B,\RR,\BB\}$-free $2$-edge-colouring.
    \item $G$ admits a $\{\R,\RR,\BB\}$-free $2$-edge-colouring.
    \item $G$ is a $\{K_1 + P_3, \mathrm{claw}\}$-free perfect graph.
    \item $G$ is either a co-bipartite graph or a cluster.
  \end{enumerate}
\end{proposition}

\begin{proof}
  The equivalence between the first two statements is trivial. Suppose that $G$
  satisfies the first statement. In particular, $G$ is an elementary graph so,
  by \cref{thm:elementary-graphs}, we know that $G$ is perfect and  claw-free.
  Also, by \cref{pro:23}, $G$ is $(K_1+P_3)$-free. Thus, the first item implies
  the third one.
  
  Now, we show that the last item implies the first one. If $G$ is co-bipartite,
  colour both cliques red and the remaining edges blue. If $G$ is a cluster,
  colour all edges red. Clearly, these are $\{\B,\RR,\BB\}$-free
  $2$-edge-colourings of $G$.
  
  For the remaining implication, recall that a result of Olariu in
  \cite{olariuIPL28} states that every perfect paw-free graph is either
  bipartite or complete multipartite. Therefore, the third item implies the
  fourth one by applying this result to the complement of a $\{K_1+P_3,
  \textrm{claw}\}$-free graph.
\end{proof}

\begin{corollary}
\label{cor:1234}
  Graphs that admit a $\{\R,\B,\RR,\BB\}$-free $2$-edge-colouring are either
  empty graphs or graphs obtained from a complete graph by deleting a matching.
\end{corollary}

\begin{proposition}
\label{pro:235}
  For a graph $G$,  the following statements are equivalent:
  \begin{enumerate}
    \item $G$ admits a $\{\B,\RR,\RB\}$-free $2$-edge-colouring.
    \item $G$ admits a $\{\R,\BB,\RB\}$-free $2$-edge-colouring.
    \item $G$ is a $\{K_1 + P_3\}$-free cograph.
    \item $G$ is a join of clusters.
  \end{enumerate}
\end{proposition}

\begin{proof}
  Suppose that a graph $G$ satisfies the first statement. By \cref{pro:23} $G$
  is $(K_1+P_3)$-free.  Also, it is not hard to observe that $P_4$ does not
  admit a $\{\B,\RR,\RB\}$-free $2$-edge-colouring. Thus, the first item implies
  the third one.

  Now, suppose that $G$ satisfies the third statement.   Since cographs are
  closed under complementation, $\overline{G}$ is a paw-free cograph.  But
  bipartite cographs are complete bipartite graphs, so, as in the proof of
  \Cref{pro:234}, $G$ is a join of clusters.

  For the remaining implication, colour the edges in each cluster red, and all
  the remaining edges blue.
\end{proof}

\begin{corollary}
\label{cor:0235}
  For a graph $G$,  the following statements are equivalent:
  \begin{enumerate}
    \item $G$ admits a $\{\IS,\B,\RR,\RB\}$-free $2$-edge-colouring.
    \item $G$ admits a $\{\IS,\R,\BB,\RB\}$-free $2$-edge-colouring.
    \item $G$ is a $3K_1$-free cograph.
    \item $G$ is a join of clusters, where each cluster contains at most two
      cliques.
    \item $G$ is a semicircular graph.
  \end{enumerate}
\end{corollary}

\begin{proof}
    \cref{lem:intersection,pro:235} imply the equivalence between the first
  three statements. The equivalence between the last three statements is stated
  in \cite{bonomoJGT61} (Theorem 8).
\end{proof}

\begin{theorem}
\label{thm:two-edges}
  The following are all the graph classes expressible by a set of
  $2$-edge-coloured graphs on three vertices and at most $2$ edges:
  \begin{enumerate}
    \item All graphs, expressed e.g.\ by $\{\R\}$.
    \item Complete multipartite graphs.
    \item Clusters, expressed by $\{\RR,\BB,\RB\}$.
    \item $3K_1$-free graphs, expressed by $\{\IS\}$.
    \item Complete graphs minus a matching, expressed by $\{\IS, \R, \B\}$. 
    \item Independent sets and complete graphs, expressed by $\{\R, \B, \RR,
        \BB, \RB\}$. 
    \item Disjoint union of at most two cliques, expressed by $\{\IS,\RR,\BB,
        \RB\}$.
    \item Complete graphs and $K_2$, expressed by $\{\IS,\R,\B,\RR,\BB,\RB\}$.
    \item Independent sets and complete graphs minus a matching, expressed by
        the set  $\{\R, \B,\allowbreak\RR,\allowbreak\BB\}$.
    \item $\{K_1+P_3, C_5\}$-free graphs, expressed by $\{\B,\RR\}$.
    \item $\{3K_1,C_5\}$-free graphs, expressed by $\{\IS,\B,\RR\}$.
    \item Join of clusters, expressed by $\{\B,\RR,\RB\}$.
    \item Join of clusters, where each cluster contains at most two cliques,
        expressed by  $\{\IS,\B,\allowbreak\RR,\RB\}$.
    \item Co-bipartite graphs and clusters, expressed by $\{\B,\RR,\BB\}$. 
    \item Co-bipartite graphs, expressed by $\{\IS,\B,\RR\}$.
    \item Elementary graphs, expressed by $\{\RR,\BB\}$.
  \end{enumerate}
\end{theorem}

Before we delve into the proof of this result, we shall obtain the following
lemma, helping us to cut the number of cases we need to process from 64 to 11.

\begin{lemma}
\label{lem:two-edges}
  If for each non-singleton subset $\calF \subseteq \{\R,\RR,\BB,\RB\}$, the
  class expressed by $\calF$ is listed in \cref{thm:two-edges}, then for each
  $\calF'$ such that $\calF' \subseteq \{\R,\RR,\BB,\RB,\B,\IS\}$, the class
  expressed by $\calF'$ is listed in \cref{thm:two-edges}.
\end{lemma}

\begin{proof}
  By symmetry, if every graph class expressed by a set $\calF \subseteq
  \{\R,\RR,\BB,\RB\}$ is listed in \cref{thm:two-edges}, then every graph class
  expressed by a set $\calF' \subseteq\{\B, \RR,\BB,\RB\}$ is listed in
  \cref{thm:two-edges}. By \cref{lem:intersection}, the class expressed by a set
  $\calF$ such that $\IS\in \calF$, is the intersection of $3K_1$-free graphs
  and the class expressed by $(\calF-\{\IS\})$. Thus, the fact that every graph
  class expressed by a set $\calF \subseteq \{\IS,\R,\RR,\BB,\RB\}$ is listed in
  \cref{thm:two-edges} will follow from the following claim --- and by symmetry,
  we will also conclude that for every graph class expressed by a set $\calF
  \subseteq \{\IS,\B,\RR,\BB,\RB\}$.

\begin{claim}
Let $\mathcal{C}$ be a class listed in \cref{thm:two-edges}. If $\mathcal{C}'$
is the intersection of $\mathcal{C}$ and $3K_1$-free graphs, then $\mathcal{C}'$
is also listed in \cref{thm:two-edges}.
\end{claim}

\begin{proof}
From 1--16, the intersections of 1,  3--9, 11, and 13--15 with $3K_1$-free
graphs are easy to obtain and to notice that these are in fact, listed in
\cref{thm:two-edges}. Also, it is not hard to notice that the intersection of
2.\ complete multipartite graphs and $3K_1$-free graphs is 5.\ complete graphs
minus a matching. Naturally, the intersection of 10.\ $\{K_1+P_3,C_5\}$-free
graphs with $3K_1$-free graphs is the class 11.\ of $\{3K_1, C_5\}$-free graphs,
and similarly, the intersection of 12.\ join of clusters ($K_1+P_3$-free
cographs, by \cref{pro:235}) with $3K_1$-free cographs is the class 13.\
$3K_1$-free cographs (join of clusters, where each cluster contains at most two
cliques \cref{cor:0235}). Finally, the intersection of elementary graphs (16.)
and $3K_1$-free graphs is the class of co-bipartite graphs (\cref{cor:J0J3J4}).
This concludes the proof of the claim.
\end{proof}

It remains to prove that for each subset $\calF' \subseteq
\{\IS,\R,\B,\RR,\BB,\RB\}$ such that $\R,\B\in \calF'$, the class expressed by
$\calF'$ is listed in \cref{thm:two-edges}. Again, by \cref{lem:intersection},
the class expressed by such a set is the intersection of complete multipartite
graphs and the class expressed by $\calF-\{\R,\B\}$.

\begin{claim}\label{cla:second}
Let $\mathcal{C}$ be a class listed in \cref{thm:two-edges}. If $\mathcal{C}'$
is the intersection of $\mathcal{C}$ and complete multipartite graphs, then
$\mathcal{C}'$ is also listed in \cref{thm:two-edges}.
\end{claim}
  
\begin{proof}
From 1--16, the intersections of 1--9 with complete multipartite graphs are easy
to obtain and to notice that these are in fact, listed in \cref{thm:two-edges}.
Clearly, $K_1+P_3$, $C_5$, and $P_4$ contain an induced $\overline{P_3}$, so,
from 10--13 their intersection with complete multipartite graphs are 2, 5, 2, 5,
respectively. Also, a co-bipartite graph is a complete multipartite if and only
if it is a complete graph minus a matching; a cluster is a complete multipartite
graph if and only if and only if it is a complete graph. Thus, both
intersections of 14 and 15 with complete multipartite graphs is the class of
complete graphs minus a matching. Finally, by \cref{lem:intersection}, the
intersection of elementary graphs and complete multipartite graph is the class
expressed by $\{\R,\B,\RR,\BB\}$. By \cref{cor:1234}, these are independent sets
and complete graphs minus a matching (9 in \cref{thm:two-edges}).
\end{proof}

By \cref{cla:second} and by the previous arguments, the lemma now follows.
\end{proof}

\begin{proof}[Proof (of \cref{thm:two-edges})]
We first justificate that for each graph class $\mathcal{C}$ listed in the
statement of \cref{thm:two-edges}, the corresponding set of $2$-edge-coloured
graphs on at most three vertices indeed expresses $\mathcal{C}$. Then we proceed
to prove that each such set $\calF$ expresses one of the graph classes from the
statement.

\begin{enumerate}
  \item Collorary of \cref{lem:trivial-constraint}.
  \item Collorary of \cref{lem:intersection}.
  \item Collorary of \cref{pro:cluster}.
  \item Trivial.
  \item Collorary of \cref{pro:comp-minus-match}.
  \item Collorary of \cref{pro:comp-minus-match}.
  \item Collorary of \cref{pro:two-cliques}
  \item Implied by \cref{lem:intersection,pro:two-cliques}.
  \item Collorary of \cref{cor:1234}.
  \item Collorary of \cref{pro:23}.
  \item Collorary of \cref{cor:023}.
  \item Collorary of \cref{pro:235}.
  \item Collorary of \cref{cor:0235}.
  \item Collorary of \cref{pro:234}.
  \item Collorar of \cref{pro:co-bip}.
  \item By definition.
\end{enumerate}

For the second part of the proof, in order to avoid considering all possible
sets of graphs on three vertices with at most two edges, we shall reduce the
number of cases. First, recall that each singleton set expresses the class of
all graphs, except for $\{\IS\}$ which expresses the class of $3K_1$-free
graphs.  Both of these classes are listed in \cref{thm:two-edges}. Furthermore,
we use \cref{lem:two-edges}

Using these tools, we conclude the proof of \cref{thm:two-edges} by listing all
non-singleton subsets of $\{\R,\RR,\BB,\RB\}$, and showing that the
corresponding expressed class is listed in \cref{thm:two-edges}. We do this in
the following two tables, where the leftmost column indicates a set $\calF
\subseteq \{\R,\RR,\BB,\RB\}$, the middle column indicates the corresponding
expressed class, and the rightmost column is the corresponding reference. 

\begin{table}[ht!]
\begin{center}
  \begin{adjustbox}{max width=0.95\textwidth}
  \rowcolors{2}{white}{gray!25}  
  \begin{tabular}{| c | l | l |}
    \hline
    $\calF$ & Class expressed by $\calF$ & Reference \\ \hline \hline
    $\RR, \BB$ & Elementary graphs (16.) & By definition \\ \hline
    $\RR, \RB$ & All graphs (1.)  & \cref{lem:trivial-constraint} \\ \hline
    $\BB, \RB$ &  All graphs (1.)  & \cref{lem:trivial-constraint}  \\ \hline
    $\RR, \BB, \RB$ & Clusters (3.)  & \cref{pro:cluster} \\ \hline
  
    \end{tabular}
  \end{adjustbox}
    \caption{Classes expressed by non-singleton sets
    $\calF\subseteq\{\RR,\BB,\RB\}$.}
    \label{tab:two-edges-1}
    \end{center}
  \end{table}

\begin{table}[ht!]
\begin{center}
  \begin{adjustbox}{max width=0.9\textwidth}
  \rowcolors{2}{white}{gray!25}  
  \begin{tabular}{| c | l | l |}
    \hline
    $\calF$ & Class expressed by $\calF$ & Reference \\ \hline \hline
    $\R, \RR$ &  All graphs (1.)  & \cref{lem:trivial-constraint}  \\ \hline
    $\R, \BB$ & $\{K_1+P_3, C_5\}$-free graphs (11.)  & \cref{pro:23} \\ \hline
    $\R, \RB$ &  All graphs (1.)  & \cref{lem:trivial-constraint}   \\ \hline
    $\R, \RR, \BB$ & Co-bipartite graphs and clusters (14.) &  \cref{pro:234} \\
    \hline
    $\R, \RR, \RB$ &  All graphs (1.) & \cref{lem:trivial-constraint} \\ \hline
    $\R, \BB, \RB$ &  Join of clusters (12.) & \cref{pro:235} \\
        \hline
    $\R, \RR, \BB, \RB$ & Clusters (3.) & \cref{pro:cluster} \\ \hline

    \end{tabular}
    \end{adjustbox}
    \caption{Classes expressed by non-singleton sets $\R\in \calF\subseteq\{\R,
    \RR,\BB,\RB\}$.}
    \label{tab:two-edges-2}
    \end{center}
    \end{table}

Finally, the proof of \cref{thm:two-edges} now follows from \cref{lem:two-edges}
and \cref{tab:two-edges-1,tab:two-edges-2}.
\end{proof}


\section{The \texorpdfstring{$\calF$}{F}-free colouring problem}
\label{sec:complex}

In this section, we study the $\calF$-free colouring problem and classify its
complexity for a variety of sets $\calF$ of $2$-edge-coloured graphs on three
vertices. These classifications are summarized in
\cref{tab:complexity-landscape}, and we prove them in a series of lemmas.

\begin{table}[ht!]
\begin{center}
  \begin{adjustbox}{max width=0.95\textwidth}
  \rowcolors{2}{white}{gray!25}  
  \begin{tabular}{|  l | c | c |}
    \hline
     Forbidden set
    & Complexity
    & Reference \\
    \hline
        $\calF$ is a trivial set
      & P
      & Remark~\ref{rmk:trivial}\\
      \hline
      $\calF$ contains $\{\R,\B\}$
      & P
      & \cref{lem:K2+K1} \\
      \hline
   $\calF$ contains $\{\IS,\RT, \BT\}$
      & P
      &  \Cref{lem:triangles+IS}\\
    \hline
     $\calF$ contains no triangle
      & P
      & \cref{lem:2-SAT} \\
      \hline
      $\calF$ contains $\{\RBB,\RRB\}$ 
      & P
      &  \cref{lem:2-SAT}\\
  \hline
    $\calF\cap {\cal T} =\{\RT,\RRB\}$
      & P
      &  \cref{lem:2-SAT} \\
  \hline
   $\calF \cap {\cal T} =\{\RT,\RBB\}$ and $\RB\in \calF$
      & P
      &  \cref{lem:linear-equations}\\
      $\calF \cap {\cal T} =\{\RT,\RBB\}$ and $\BB,\RR\in \calF$
      & P
      &  \cref{lem:linear-equations}\\
  \hline
    $\calF\subseteq \{\IS, \R, \B, \RR, \RB, \RT, \RRB\}$
      & P
      & \cref{lem:consistency}\\
      \hline  
       $\calF\subseteq \{\IS, \RR, \BB, \BT, \RT, \RBB, \RRB\}$
      & P
      &  \cref{lem:elementary-cases}\\
      \hline
    $\calF = \{\BT, \RT\}$
      & NP-complete
      & \cite{burrMRT1990} \\
      $\calF = \{\BT, \RT, \RBB\}$
      & NP-complete
      &  \Cref{lem:triangles}\\
  \hline
  
  \end{tabular}
  \end{adjustbox}
  \caption{Up to colour symmetry, the complexity landscape of the $\calF$-free
  colouring problem settled in Section~\ref{sec:complex}. In this table,  $\cal
  T$ denotes the set of $2$-edge-coloured triangles on three vertices, i.e.,
  ${\cal T}=\{\BT,\RT,\RBB,\RRB\}$.}
  \label{tab:complexity-landscape}
\end{center}
\end{table}

\begin{lemma}\label{lem:K2+K1}
    Let $\calF$ be a set of $2$-edge-coloured graphs on three vertices. If
    $\B,\R\in \calF$, then the $\calF$-free colouring problem is polynomial-time
    solvable.
\end{lemma}
\begin{proof}
    Since $\B,\R\in \calF$, any graph that admits an $\calF$-free colouring is a
    complete multipartite graph. The remaining forbidden coloured graphs in
    $\calF$ determine the maximum number of vertices in each class of the
    partition, and the maximum number of classes. It is straightforward to
    observe that in any such case, the $\calF$-free colouring problem can be
    solved in polynomial time.
\end{proof}

The following lemma sums up some of the easy cases which mostly follow from the
previous characterizations we proved.

\begin{lemma}\label{lem:elementary-cases}
    Let $\calF$ be a set of $2$-edge-coloured graphs on three vertices. If
    $\calF \subseteq \{\IS,\BB,\RR,\allowbreak\BT,\RT,\RBB,\RRB\}$, then the
    $\calF$-free colouring problem is polynomial-time solvable.
\end{lemma}
\begin{proof}
    If $\calF\subseteq \{\BB,\RR,\BT,\RT,\RRB,\RBB\}$, then the $\calF$-free
    colouring problem is poly\-no\-mi\-al-time solvable by our structural
    characterizations in Sections~\ref{sec:elementary}
    and~\ref{sec:further-elementary} (see also \Cref{fig:landscape}). And if
    $\IS\in \calF$, we can first verify if the input graph $G$ is $\IS$-free,
    and we then run the algorithm for the $(\calF\setminus\{\IS\})$-free
    colouring problem.
\end{proof}

The following three lemmas describe cases of $\calF$ such that the corresponding
$\calF$-free colouring problem is solvable (in polynomial time) by 2-SAT, linear
equations of  $\mathbb Z_2$, and by the so-called consistency, respectively.

\begin{lemma}\label{lem:2-SAT}
    Let $\calF$ be a set of $2$-edge-coloured graphs on three vertices. In any
    of the following cases, the $\calF$-free colouring problem reduces to
    $2$-SAT.
    \begin{itemize}
        \item $\calF$ contains no triangle.
        \item $\calF$ contains all triangles.
        \item the subset of triangles of $\calF$ is $\{\RT,\RRB\}$.
        \item the subset of triangles of $\calF$ is $\{\BT,\RBB\}$.
    \end{itemize}
\end{lemma}
\begin{proof}
    The case when $\calF$ contains no triangle is immediate: for every edge $e$,
    we include a variable $v_e$, and by interpreting blue as \texttt{true} and
    red as \texttt{false}, it is straightforward to construct a $2$-SAT instance
    $\phi_\calF$ such that $\phi_\calF$ is true if and only if $G$ admits an
    $\calF$-free colouring. Also, if $\calF$ contains all triangles, then on
    input $G$ we first verify whether $G$ is triangle-free, and we then reduce
    to $2$-SAT as in the previous case. For the remaining cases, let $\calF':=
    \calF\setminus\{\BT,\RT,\RBB,\RRB\}$. The case when the subset of triangles
    of $\calF$ is $\{\RT,\RRB\}$, we consider  the formula $\phi_{\calF'}$ and
    for each pair of different edges $e$ and $f$ is a common triangle, we add
    the clause $(v_e \lor v_f)$ to $\phi_{\calF'}$. In the case when the subset
    of triangles of $\calF$ is $\{\RRB,\RBB\}$, for each pair of edges $e$ and
    $f$ on a common triangle, we add the conjunct $(v_e\lor \lnot v_f)\land
    (\lnot v_e \lor v_f)$.
\end{proof}

\begin{lemma}\label{lem:linear-equations}
    Let $\calF$ be a set of $2$-edge-coloured graphs on three vertices. If
    $\RB\in \calF$ or $\RR,\BB\in\calF$, and the set of triangles in $\calF$ is
    $\{\RT,\RBB\}$ or $\{\BT,\RRB\}$, then the $\calF$-free colouring problem
    can be solved via linear equations of $\mathbb Z_2$.
\end{lemma}
\begin{proof}
    We consider the case when $\RB ,\RT,\RBB\in \calF$, and as noted before,
    after some pre-processing, we can also assume that $\IS\not\in\calF$. Now
    notice that, by interpreting blue as $0$ and red as $1$, it is
    straightforward to observe that any $\calF$-free colouring of $G$ satisfies
    the system of linear equation (over $\mathbb Z_2$) obtained by adding a
    variable $x_e$ for every edge $e\in E(G)$; by adding for each pair of edges
    $e$ and $f$ inducing a $P_3$ the equation $x_e + v_f = 0$, and, finally, by
    adding for each triple of edges $e,f,g$ inducing a triangle the equation
    $x_e + x_f + x_g = 1$. Now notice that the remaining forbidden coloured
    graphs extend this system of linear equations so that the solutions to the
    extended system are in one-to-one correspondence with the $\calF$-free
    colourings of $G$: if $\B\in \calF$, add the equation $x_e = 1$ for every
    edge $e$ that belongs to some copy of $K_1 + K_2$ in $G$; if $\BB \in
    \calF$, then add the equation $x_e = 1$ for every edge that belongs to an
    induced path on $3$ vertices (recall that $\RB\in \calF$); symmetrically, if
    $\R\in \calF$ or $\RR\in \calF$, we consider the equations $x_e = 0$ for
    each edge $e$ that belongs to an induced $K_1 + K_2$ or an induced $P_3$,
    respectively. The remaining cases follow similarly.
\end{proof}

\begin{lemma}\label{lem:consistency}
    Let $\calF$ be a set of $2$-edge-coloured graphs on three vertices. If
    $\calF\subseteq \{\IS, \R, \B, \allowbreak \RR, \RB, \RT, \RRB\}$ or
    $\calF\subseteq \{\IS, \R, \B, \BB, \RB, \BT, \RBB\}$, then the $\calF$-free
    colouring problem can be solved in polynomial time.
\end{lemma}
\begin{proof}
    As previously noted, we can assume that $\IS\not\in \calF$, and by
    Lemma~\ref{lem:K2+K1} we restrict to the cases when $\R\not\in\calF$ or
    $\B\not\in \calF$. Also, if $\calF$ contains no red monochromatic or no blue
    monochromatic graph, the $\calF$-free colouring problem is trivial. Notice
    that up to colour symmetry, the remaining sets satisfy $\R\in \calF$ and
    $\calF\subseteq\{\R,\BB, \RB, \BT, \RBB\}$. In any such case, it is
    straightforward to observe that $G$ admits an $\calF$-free colouring if and
    only if the following colouring of $G$ is $\calF$-free: colour every edge of
    $G$ that belongs to an induced copy of $K_1 +K_2$ blue, and colour the
    remaining edges red.
\end{proof}

\begin{lemma}\label{lem:triangles}
    If $\calF = \{\BT,\RT,\RRB\}$, or $\calF = \{\BT,\RT,\RBB\}$, then the
    $\calF$-free colouring problem is $\NP$-complete. 
\end{lemma}
\begin{proof}
    Both cases are symmetric, we consider the case $\calF = \{\BT,\RT,\RBB\}$,
    and reduce from positive 1-in-3 SAT to the $\calF$-free colouring problem.
    Furthermore, we represent the input of 1-in-3 SAT as a boolean formula in
    CNF with clauses of size 3. 

    For a given formula $\Phi$ which is an instance of the aforementioned
    SAT-problem, we shall construct a graph $G$ such that $\Phi$ is satisfiable
    (in the sense of this particular SAT-problem) if and only if $G$ is admits a
    $\{ \BT, \RT, \RBB \}$-free 2-edge-colouring.

    For each clause, we create a new triangle in $G$, vertex-disjoint with all
    other vertices of $G$. Edges of this triangle will correspond to variables.
    We now need to introduce a copying gadget to ensure that all edges
    corresponding to a given variable will have the same colour. We distinguish
    two types of copying gadgets: an \emph{atom}, which is a copying gadget
    between two different clauses, and a \emph{block}, which is a copying gadget
    within the same clause (resolving, e.g., cases like $x \lor x \lor y$). Note
    that a block consists of two atoms sharing an edge; see
    \Cref{fig:copying_gadgets}.

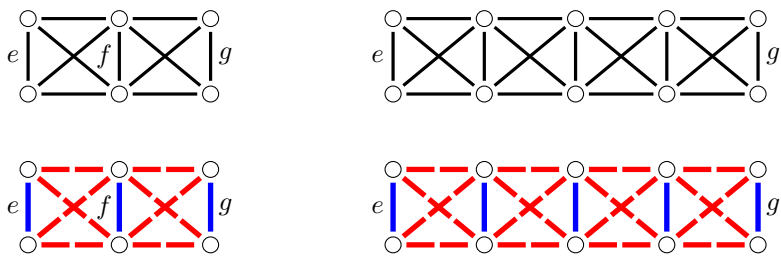
\begin{figure}[htb!]
\centering
\begin{tikzpicture}
  \begin{scope}
  \node at (-1.4,0.5){$e$};
  \node at (-0.2,0.5){$f$};
   \node at (1.4,0.5){$g$};
   \node [vertex] (1) at (-1.2,0){};
    \node [vertex] (2) at (-1.2,1){};
    \node [vertex] (3) at (0,0){};
    \node [vertex] (4) at (0,1) {};
    \node [vertex] (5) at (1.2,0) {};
    \node [vertex] (6) at (1.2,1) {};
    
    \foreach \i/\j in {1/2, 1/3, 1/4, 2/3, 2/4, 3/4, 3/5, 3/6, 4/5, 4/6, 5/6}
    \draw [edge] (\i) to (\j);
  \end{scope}
  
  \begin{scope}[xshift=6cm]
    \node at (-2.6,0.5){$e$};
    \node at (2.6,0.5){$g$};
    \node [vertex] (a) at (-2.4,0) {};
    \node [vertex] (b) at (-2.4,1) {};
    \node [vertex] (1) at (-1.2,0){};
    \node [vertex] (2) at (-1.2,1){};
    \node [vertex] (3) at (0,0){};
    \node [vertex] (4) at (0,1) {};
    \node [vertex] (5) at (1.2,0) {};
    \node [vertex] (6) at (1.2,1) {};
    \node [vertex] (c) at (2.4,0) {};
    \node [vertex] (d) at (2.4,1) {};
    
    \foreach \i/\j in {a/b, a/1, a/2, b/1, b/2, 1/2, 1/3, 1/4, 2/3, 2/4, 3/4,
      3/5, 3/6, 4/5, 4/6, 5/6, c/d, c/5,c/6,d/5, d/6}
      \draw [edge] (\i) to (\j);
  \end{scope}
  

  \begin{scope}[yshift = -2cm]
    \node at (-1.4,0.5){$e$};
    \node at (-0.2,0.5){$f$};
    \node at (1.4,0.5){$g$};
    \node [vertex] (1) at (-1.2,0){};
    \node [vertex] (2) at (-1.2,1){};
    \node [vertex] (3) at (0,0){};
    \node [vertex] (4) at (0,1) {};
    \node [vertex] (5) at (1.2,0) {};
    \node [vertex] (6) at (1.2,1) {};
    
    \foreach \i/\j in {1/3, 1/4, 2/3, 2/4, 3/5, 3/6, 4/5, 4/6}
      \draw [redE] (\i) to (\j);
    
    \foreach \i/\j in {1/2, 3/4, 5/6}
      \draw [blueE] (\i) to (\j);
  \end{scope}
  
  \begin{scope}[xshift=6cm, yshift = -2cm]
    \node at (-2.6,0.5){$e$};
    \node at (2.6,0.5){$g$};
    \node [vertex] (a) at (-2.4,0) {};
    \node [vertex] (b) at (-2.4,1) {};
    \node [vertex] (1) at (-1.2,0){};
    \node [vertex] (2) at (-1.2,1){};
    \node [vertex] (3) at (0,0){};
    \node [vertex] (4) at (0,1) {};
    \node [vertex] (5) at (1.2,0) {};
    \node [vertex] (6) at (1.2,1) {};
    \node [vertex] (c) at (2.4,0) {};
    \node [vertex] (d) at (2.4,1) {};
    
    \foreach \i/\j in {a/1, a/2, b/1, b/2, 1/3, 1/4, 2/3, 2/4, 3/5, 3/6, 4/5,
                       4/6, c/5,c/6,d/5, d/6}
      \draw [redE] (\i) to (\j);
    
    \foreach \i/\j in {a/b, 1/2, 3/4,  5/6, c/d}
      \draw [blueE] (\i) to (\j);
  \end{scope}
  
\end{tikzpicture}

\caption{An atom on the top left and a block on the top  right. On the bottom
left (resp.\ right), the unique $\{ \BT, \RT, \RBB \}$-free  $2$-edge-colouring
of the atom (resp.\ block) where $e$ is coloured blue}
\label{fig:copying_gadgets}
\end{figure}

First, we describe possible 2-edge-colourings for atoms. An atom is the graph
resulting from two $K_4$ by identifying one edge, say $f$ (see
\Cref{fig:copying_gadgets} for edge labels). If the edge $e$ is blue, then all
triangles containing $e$ must have all other edges red. Thus $f$ must be blue,
otherwise we obtain a red triangle. By a similar argument we can conclude that
$g$ must be also blue. On the other hand, if $e$ is red, then $f$ and also $g$
must be red, again by similar reasoning. Since a block consists of two atoms
sharing an edge, we can conclude that edge $e$ in every block has the same
colour as $f$ in every $\{ \BT, \RT, \RBB \}$-free 2-edge-colouring.

For each pair of occurrences of a given variable in different clauses, we create
an atom so that edge $e$ is identified with one occurrence and edge $g$ with the
other one. The properties of the atom ensure that the colour of the edges and
thus their truth values are the same. If a variable occurs twice in one clause,
we create a block so that that edge $e$ is identified with one occurrence and
edge $g$ with the other one. Observe that we cannot use atoms in this case,
since we would create multi-edges. 

Notice that we do not introduce any new triangles (besides, of course, those
introduced in each gadget) by putting all the copying gadgets together. Since
the distance between any vertex incident with $e$ and any vertex incident with
$f$ is at least two and there are no edges between two different blocks or
atoms, we cannot create any new triangle in the construction. Thus each triangle
in $G$ corresponds to exactly one clause in $\Phi$.

Finally, on the one hand, let $G$ be the graph we have just constructed for the
formula $\Phi$. If there is a $\{ \BT, \RT, \RBB \}$-free 2-edge-colouring for
$G$, then each triangle has exactly one blue edge, meaning that exactly one
literal is true. Furthermore, all occurrences of the same variable are valued
consistently because of the properties of atoms and blocks. Therefore, we are
able to extract the corresponding satisfying valuation of variables of $\Phi$.
On the other hand, suppose that $\Phi$ is satisfiable and take a satisfying
assignment $v_{\Phi}$. If $v_{\Phi}(x)=1$ then colour all edges corresponding to
occurrences of $x$ blue, and colour them red otherwise. Since in each clause
there is exactly one variable set to true and all occurrences of a variable are
coloured by the same colour, it is always possible to complete the colouring of
$G$ to a $\{ \BT, \RT, \RBB \}$-free colouring of atoms and blocks. This
completes the reduction.
\end{proof}

\section{A uniform reduction to boolean CSPs}
\label{sec:uniform}

The reader familiar with constraint satisfaction problems (CSPs) may have
realized that the reductions in Lemmas~\ref{lem:2-SAT},
and~\ref{lem:linear-equations}, are reductions to boolean CSPs, and the
algorithm from Lemma~\ref{lem:consistency} is a simple consistency checking
which also solves certain boolean CSPs. In this section, we
present a uniform reduction from the $\calF$-free colouring problem to boolean
CSPs (adapted from~\cite{bodirskyArXiv} to our context).

A \textit{relational signature} $\tau$ is a set of relation symbols $R, S,\dots$
each equipped with a positive integer called its \textit{arity}. A
\textit{$\tau$-structure} $\mathbb A$ consists of a vertex set $V(\mathbb A)$,
and for each $R \in \tau$ of arity $r$, of an $r$-ary relation $R(\mathbb
A)\subseteq V(\mathbb A)^r$ called the \textit{interpretation} of $R$ in
$\mathbb A$. If $\mathbb A$ has a two-element vertex set, we call it a
\textit{boolean structure}. Given a pair of $\tau$-structures $\mathbb A,
\mathbb B$ a \textit{homomorphism} $f\colon \mathbb A\to \mathbb B$ is a
function $f\colon V(\mathbb A)\to V(\mathbb B)$ such that  $(a_1,\dots, a_r) \in
R(\mathbb A)$ implies that $(f(a_1),\dots, f(a_r)) \in R(\mathbb B)$ for each $R
\in \tau$ of arity $r$. In this case we write $\mathbb A\to \mathbb B$, and we
denote by $\CSP(\mathbb B)$ the class of finite structures $\mathbb A$ such that
$\mathbb A\to \mathbb B$. The \textit{constraint satisfaction problem} (CSP)
with template $\mathbb B$ consists of deciding whether $\mathbb A\in
\CSP(\mathbb B)$ on input structure $\mathbb A$. Note that in this setting, a
graph $G$ can be regarded as an $\{E\}$-structure where $E$ is a relation symbol
of arity $2$, and the edge set of $G$ corresponds to the interpretation of $E$
in $G$. So, given a graph $H$, the constraint satisfaction problem with template
$H$ is essentially the $H$-colouring problem --- formally, the input space of
$\CSP(H)$ consists of all finite digraphs. In particular, $\CSP(K_k)$ is the
problem of deciding if the underlying graph of an input digraph $D$ is
$k$-colourable. 

Consider a $2$-edge-coloured graph $H$ with vertex set $\{1,\dots, n\}$ and
non-empty edge set $E(G)$ with lexicographical ordering $e_1 < \dots < e_m$,
i.e., $e_i < e_j$ for $e_i = ab$ and $e_j = cd$ if $\min\{a,b\} < \min\{c,d\}$
or if $\min\{a,b\} = \min\{c,d\}$ and $\max\{a,b\} < \max\{c,d\}$. Clearly, we
can code the $2$-edge-colouring of $H$ with a tuple $c_H\in\{r,b\}^m$ where
$(c_H)_i= r$ if the edge $e_i\in E(G)$ is coloured red, and $(c_H)_i = b$ if the
edge $e_i$ is coloured blue. Using this simple representation of
$2$-edge-colourings of labelled graphs, we will reduce the $\mathcal F$-free
$2$-edge-colouring problem to a CSP of a boolean structure as follows.

For a finite set $\calF$ of $2$-edge-coloured graphs, let $\ell(\calF)$ be the
set of underlying labelled graphs of $\calF$. Equivalently, $\ell(\calF)$
consists of  graphs with vertex set $\{1,\dots, |V(F)|\}$ isomorphic to the
underlying graph of some $F\in \calF$. The signature of the boolean structure
$\mathbb B_\calF$ consists of a relation symbol $R_L$ for each $L\in
\ell(\calF)$, and the arity of $R_L$ is $|E(L)|$  if $L$ has at least one edge,
and if $L$ is an edgeless graph, then the arity of $R_L$ is $1$; equivalently,
the arity of $R_L$ is $\max\{1,|E(L)|\}$. The boolean structure $\mathbb
B_\calF$ has vertex set $\{r,b\}$. For each $L\in \ell(\calF)$, the
interpretation of $R_L$ in $\mathbb B_\calF$ consists of all tuples $c_{L'}$
where $L'$ is an $\mathcal F$-free $2$-edge-colouring of $L$. Notice that for
each edgeless $L$, the interpretation of $R_L$ in $\mathbb B_\calF$ is empty.

For instance, if $\mathcal F$ consists of both monochromatic paths on three
vertices, then the signature of $\mathbb B_\mathcal F$ consists of three binary
relational symbols $R_1,R_2,R_3$ one for each labeled path on three vertices.
Suppose that $R_1$ corresponds to the labeled path with edge set $\{12,23\}$, so
the interpretation of $R_1(\mathbb B_\mathcal F)$ consists of the tuples $(r,b)$
and $(b,r)$.  Note that in this case, the interpretation of the three binary
relations are the same but in general this might not be the case.

\begin{lemma}\label{lem:boolean-CSP}
  For any finite set of $2$-edge-coloured graphs $\calF$, there is a
  polynomial-time reduction from the $\calF$-free $2$-edge-colouring problem to
  $\CSP(\mathbb B_\calF)$.
\end{lemma}

\begin{proof}
On a given input graph $G$, we construct a structure $\mathbb B_G$ such that
$\mathbb B_G\in \CSP(\mathbb B_\calF)$ if and only if $G$ admits an $\calF$-free
$2$-edge-colouring. First, if $\calF$ contains an empty graph $L$ on $k$
vertices, verify whether $G$ contains an independent set on at least $k$
vertices. If yes, construct $\mathbb B_G$ consisting of a single vertex $v$ and
$R_L(\mathbb B_G) = \{v\}$. Otherwise,  we assume that $G$ is $L$-free for every
edgeless graph $L\in \calF$. Consider an enumeration $v_1,\dots, v_n$ of $V(G)$.
For each edge $e \in E(G)$, there is a vertex $x_e \in V(\mathbb B_G)$. For
every $R_H$ with $H\in \ell(\calF)$ on vertex set $\{1,\dots, k\}$ and at least
one edge, we construct the relation $R_H(\mathbb B_G)$ as follows. For every $k$
vertices $v_{i_1},\dots, v_{i_k}$ of $G$ with $i_1 < \dots < i_k$, let $e_1 <
\dots < e_m$ be the lexicographical ordering of the edges of $G[\{ v_1, \dots,
v_k \}]$. If the mapping $v_i \mapsto i$ is an isomorphism from $G[\{v_1,\dots,
v_k\}]$ to $H$, then we add the tuple $(x_{e_1},\dots, x_{e_m})$ to the
interpretation of $R_H$ in $\mathbb B_G$. It is straightforward to observe that
$\mathbb B_G \in \CSP(\mathbb B_\calF)$ if and only if $G$ admits an
$\calF$-free $2$-edge-colouring. 
\end{proof}

Besides providing a more systematic and unified way to obtain polynomial-time
algorithms for number of classes we encountered in this section, this lemma also
allows us to exhibit limitations of such immediate reductions: using Schaefer's
classification of the complexity of boolean CSPs, we can determine in which
cases the ``natural'' reduction from the $\calF$-free to a boolean CSP, is just
a reduction to an $\NP$-complete problem. In Appendix~\ref{ap:complex} we
present this classification for forbidden sets of $2$-edge-coloured graphs.

In particular, when  $\calF = \{\BT,\RT,\RRB\}$, then the boolean CSP
$\CSP(\mathbb B_\calF)$ encodes positive 1-IN-3 SAT, and in
Lemma~\ref{lem:triangles} we show that positive 1-IN-3 reduces back to the
$\calF$-free colouring problem. It makes sense to ask for which sets $\calF$ it
is the case that the $\CSP(\mathbb B_\calF)$ and the $\calF$-free colouring
problem are polynomial-time equivalent. Further notice that whenever this is the
case, the $\calF$-free colouring problem is in P or NP-complete (and it cannot
be an NP-intermediate problem).

\begin{corollary}\label{cor:triangles}
    For every set of $2$-edge-coloured triangles, the $\calF$-free
    $2$-edge-colouring problem and $\CSP(\mathbb B_\calF)$ are polynomial-time
    equivalent.
\end{corollary}

\begin{proof}
The $\calF$-free $2$-edge-colouring problem reduces in polynomial time to
$\CSP(\mathbb B_\calF)$ (\Cref{lem:boolean-CSP}). As mentioned before, if
$\calF$ is a trivial set, then there is nothing left to prove, and it follows
from the NP-hardness of the $\{\RT, \BT\}$-free colouring problem that the claim
is true for $\calF = \{\RT, \BT\}$. Finally, the cases when $\calF = \{\RT, \BT,
\RRB\}$ and $\calF = \{\RT, \BT, \RBB\}$ follow via Lemma~\ref{lem:triangles}.
\end{proof}

In a similar context Bodirsky and Guzm\'an-Pro~\cite{bodirskyArXiv} showed that
for every finite set of tournaments $\mathcal T$, there is a boolean structure
$\mathbb B_\mathcal T$ such that the $\mathcal T$-free orientation problem and
$\CSP(\mathbb B_\mathcal T)$ are polynomial-time equivalent. In particular, this
implies a P vs. NP-complete dichotomy for the $\mathcal T$-free orientation
problem. As previously mentioned, the construction of $\mathbb B_\calF$ for
forbidden $2$-edge-coloured graphs is analogous to the construction of $\mathbb
B_\mathcal T$ for forbidden tournaments $\mathcal T$. Thus, it makes sense to
ask the following question. 

\begin{question}\label{qst:complete-obstructions}
    Is it true that for every finite set $\calF$ of $2$-edge-coloured complete
    graphs, the $\calF$-free $2$-edge-colouring problem and $\CSP(\mathbb
    B_\calF)$ are polynomial-time equivalent?
\end{question}

In particular, a positive answer to this question implies that $\calF$-free
colouring problems have a P versus NP-complete dichotomy whenever $\calF$ is a
finite set of $2$-edge-coloured complete graphs. Actually, such problems can be
expressed in a logic called GMSNP for which such a P versus NP-complete
dichotomy is conjectured (we shall not go into the detailed definitions here, we
refer the reader to \cite{barsukovLNCS23,bienvenuTODS14}). However,
when we consider arbitrary finite set $\calF$ of (not necessarily complete)
$2$-edge-coloured graphs, we land in an extension of GMSNP that allows the use
of inequality if its formulas, and it is known that this logic captures
NP-intermediate problems (unless P = NP)~\cite{barsukovLNCS23,federSIAM28}. This
suggest that there might be no P versus NP-complete dichotomy for $\calF$-free
colouring problems.

\begin{question}\label{qst:NP-intermediate}
    Is there a finite set $\calF$ of $2$-edge-coloured graphs such that the
    $\calF$-free colouring problem is $\NP$-intermediate (assuming $\PO\neq
    \NP$)?
\end{question}


\section{Larger patterns and more colours}
\label{sec:general}

As noted above, forbidden edge-coloured-pattern problems have been already
studied for general relational structures, and in the context of CSPs and
computational complexity~\cite{madelaineLMCS5}. By staying in the realm of
structural graph theory and specifically by considering undirected loopless
graphs, there are two interesting natural generalizations to the work presented
in the previous section. Either one can consider forbidden $2$-edge-coloured
graphs on more vertices, or forbid $k$-edge-coloured graphs with $k \ge 3$. In
this section, we make some initial observations regarding forbidden sets with
$2$-edge-coloured graphs on more than three vertices, and forbidden sets with
$k$-edge-coloured graphs. In particular, we will notice how some of our previous
results naturally generalize to this broader context.

Given a positive integer $k$, we denote by $\overline{B}_k$ the set of
$2$-edge-coloured graphs on $k$ vertices without any blue edge. For instance,
$\overline{B}_1 = \{K_1\}$, $\overline{B}_2$ contains $2K_1$ and ${}^r K_2$, and
$\overline{B}_3 = \{\IS, \R, \RR, \RT\}$. Similar to \cref{pro:co-bip}, the
following example shows that co-$k$-partite graphs admit a natural expression by
means of $2$-edge-coloured graphs.

\begin{proposition}
\label{prop:co-k-partite}
  For a graph $G$ and a positive integer $k$, the following statements are
  equivalent:
  \begin{enumerate}
    \item $G$ is a co-$k$-partite graph.
    \item $G$ admits a ($\{\BB, \RBB\}\cup \overline{B}_{k+1}$)-free
      $2$-edge-colouring.
    \item The edge set of $G$ admits a partition $(B,R)$ such that $B$ induces a
      spanning cluster of $G$, and any subgraph of $G$ induced by $k+1$ vertices
      contains an edge of $B$.
  \end{enumerate}
\end{proposition}

\begin{proof}
It is immediate to see that the third item implies that the vertex set of $G$
admits a partition $(X_1,\dots, X_j)$ into $j$ complete subgraphs where $j \leq
k$. Thus, the third item implies the first one.

To see that the first statement implies the second one, consider a  partition
$(X_1,\dots, X_j)$ of $V(G)$ into $j$ cliques where $j \leq k$. Colour an edge
$e$ of $G$ blue if it belongs to $E(X_i)$, for some $i \in \{1, \dots, j\}$, and
red otherwise. Since $j \le k$, any $k+1$ vertices induce at least one blue
edge, hence this colouring of $G$ is $\overline{B}_k$-free. Clearly, by removing
the red edges from $G$, we obtain a cluster and so, this colouring is also a
$\{\BB, \RBB\}$-free colouring of $G$.

Finally, to see that the second item implies the third one, notice that the
($\{\BB,\allowbreak \RBB\}\cup\allowbreak \overline{B}_{k+1}$)-free
$2$-edge-colouring of $G$ defines a partition $(B,R)$ of $E(G)$ that satisfies
the third statement.
\end{proof}

As hinted above, we can recover \cref{pro:co-bip} as a particular instance of
\cref{prop:co-k-partite}.   For $k = 2$, \cref{prop:co-k-partite} implies that a
graph $G$ is co-bipartite if and only if it admits a $\{\IS, \R, \RR, \BB, \RT,
\RBB\}$-free $2$-edge-colouring.

We now consider an example of a family of graphs classes expressible by
forbidden $k$-edge-coloured graphs, with $k \ge 2$. This example is an extension
of the proposed characterization of line graphs of bipartite graphs
(\cref{prop:J3J4J8J9}) to line graphs of $k$-partite graphs. Let $k$ be a
positive integer such that $k \ge 2$.  For an integer $l$, with $l \le k$, we
abuse nomenclature and say that a vertex $x$ is incident with $l$ colours if $x$
is incident with $l$ edges of pairwise different colours. A triangle is
\emph{rainbow} if its three edges have different colours. We denote by $L_k$ the
set of $k$-coloured graphs that consists of the following graphs: each
monochromatic $P_3$, each non-monochromatic and non-rainbow triangle, and each
$k$-edge-coloured graph on $4$ vertices that contains a vertex incident with $3$
colours (but does not contain any of the previous mentioned paths and
triangles). In particular, $L_2 = \{\RR, \BB, \RRB, \RBB\}$, and in
\cref{fig:3-colours} we illustrate the $3$-edge-coloured graphs in $L_3$.

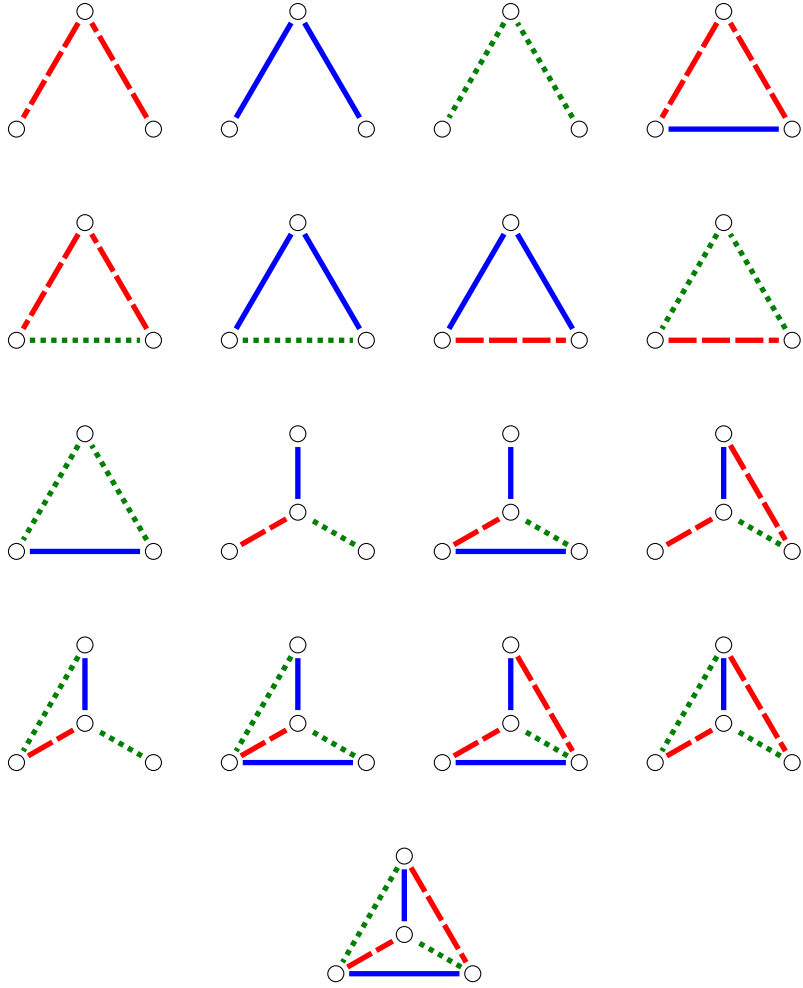
\begin{figure}[ht!]
\centering
\begin{tikzpicture}
\begin{scope}[scale=0.8]
  \begin{scope}
    \foreach \i in {0,1,2}
      \draw ({(360/3)*\i + 90}:1.3) node(\i)[vertex]{};

    \draw [redE] (0) to (1);
    \draw [redE] (2) to (0);
  \end{scope}
  
  \begin{scope}[xshift=3.5cm]
    \foreach \i in {0,1,2}
      \draw ({(360/3)*\i + 90}:1.3) node(\i)[vertex]{};

    \draw [blueE] (0) to (1);
    \draw [blueE] (2) to (0);
  \end{scope}
  
  \begin{scope}[xshift=7cm]
    \foreach \i in {0,1,2}
      \draw ({(360/3)*\i + 90}:1.3) node(\i)[vertex]{};

    \draw [greenE] (0) to (1);
    \draw [greenE] (2) to (0);
  \end{scope}
  
  \begin{scope}[xshift=10.5cm]
    \foreach \i in {0,1,2}
      \draw ({(360/3)*\i + 90}:1.3) node(\i)[vertex]{};

    \draw [redE]  (0) to (1); \draw [blueE] (1) to (2); \draw [redE] (2) to (0);
  \end{scope}

  \begin{scope}[yshift=-3.5cm]
    \foreach \i in {0,1,2}
      \draw ({(360/3)*\i + 90}:1.3) node(\i)[vertex]{};
    
    \draw [redE]   (0) to (1);
    \draw [greenE] (1) to (2);
    \draw [redE]   (2) to (0);
  \end{scope}
  
  \begin{scope}[xshift=3.5cm, yshift=-3.5cm]
    \foreach \i in {0,1,2}
      \draw ({(360/3)*\i + 90}:1.3) node(\i)[vertex]{};
    
    \draw [blueE]  (0) to (1);
    \draw [greenE] (1) to (2);
    \draw [blueE]  (2) to (0);
  \end{scope}
  
  \begin{scope}[xshift=7cm,yshift=-3.5cm]
    \foreach \i in {0,1,2}
      \draw ({(360/3)*\i + 90}:1.3) node(\i)[vertex]{};
    
    \draw [blueE] (0) to (1);
    \draw [redE]  (1) to (2);
    \draw [blueE] (2) to (0);
  \end{scope}
  
  \begin{scope}[xshift=10.5cm,yshift=-3.5cm]
    \foreach \i in {0,1,2}
      \draw ({(360/3)*\i + 90}:1.3) node(\i)[vertex]{};
    
    \draw [greenE] (0) to (1);
    \draw [redE]   (1) to (2);
    \draw [greenE] (2) to (0);
  \end{scope}

  \begin{scope}[yshift=-7cm]
    \foreach \i in {0,1,2}
      \draw ({(360/3)*\i+ 90}:1.3) node(\i)[vertex]{};

      \draw [greenE] (0) to (1);
      \draw [blueE]  (1) to (2);
      \draw [greenE] (2) to (0);
  \end{scope}
  
  \begin{scope}[xshift=3.5cm,yshift=-7cm]
    \foreach \i in {0,1,2}
      \draw ({(360/3)*\i+ 90}:1.3) node(\i)[vertex]{};
    \node(3)[vertex]{};
    
    \draw [blueE]  (0) to (3);
    \draw [redE]   (1) to (3);
    \draw [greenE] (2) to (3);
  \end{scope}
  
   \begin{scope}[xshift=7cm, yshift=-7cm]
    \foreach \i in {0,1,2}
      \draw ({(360/3)*\i+ 90}:1.3) node(\i)[vertex]{};
    \node(3)[vertex]{};
    
    \draw [blueE]  (0) to (3);
    \draw [redE]   (1) to (3);
    \draw [greenE] (2) to (3);
    \draw [blueE]  (1) to (2); 
  \end{scope}
  
  \begin{scope}[xshift=10.5cm, yshift=-7cm]
    \foreach \i in {0,1,2}
      \draw ({(360/3)*\i+ 90}:1.3) node(\i)[vertex]{};
    \node(3)[vertex]{};
    
    \draw [blueE]  (0) to (3);
    \draw [redE]   (1) to (3);
    \draw [greenE] (2) to (3);
    \draw [redE]   (0) to (2); 
  \end{scope}
  
  \begin{scope}[yshift=-10.5cm]
    \foreach \i in {0,1,2}
      \draw ({(360/3)*\i+ 90}:1.3) node(\i)[vertex]{};
    \node(3)[vertex]{};
    
    \draw [blueE]  (0) to (3);
    \draw [redE]   (1) to (3);
    \draw [greenE] (2) to (3);
    \draw [greenE] (0) to (1); 
  \end{scope}
  
  \begin{scope}[xshift=3.5cm,yshift=-10.5cm]
    \foreach \i in {0,1,2}
      \draw ({(360/3)*\i+ 90}:1.3) node(\i)[vertex]{};
    \node(3)[vertex]{};
    
    \draw [blueE]  (0) to (3);
    \draw [redE]   (1) to (3);
    \draw [greenE] (2) to (3);
    \draw [greenE] (0) to (1);
    \draw [blueE]  (1) to (2); 
  \end{scope}
  
   \begin{scope}[xshift=7cm, yshift=-10.5cm]
    \foreach \i in {0,1,2}
      \draw ({(360/3)*\i+ 90}:1.3) node(\i)[vertex]{};
    \node(3)[vertex]{};
    
    \draw [blueE]  (0) to (3);
    \draw [redE]   (1) to (3);
    \draw [greenE] (2) to (3);
    \draw [blueE]  (1) to (2);
    \draw [redE]   (0) to (2); 
  \end{scope}
  
  \begin{scope}[xshift=10.5cm, yshift=-10.5cm]
    \foreach \i in {0,1,2}
      \draw ({(360/3)*\i+ 90}:1.3) node(\i)[vertex]{};
    \node(3)[vertex]{};
    
    \draw [blueE]  (0) to (3);
    \draw [redE]   (1) to (3);
    \draw [greenE] (2) to (3);
    \draw [redE]   (0) to (2);
    \draw [greenE] (0) to (1); 
  \end{scope}
  
  \begin{scope}[xshift = 5.25cm, yshift=-14cm]
    \foreach \i in {0,1,2}
      \draw ({(360/3)*\i+ 90}:1.3) node(\i)[vertex]{};
    \node(3)[vertex]{};
    
    \draw [blueE]  (0) to (3);
    \draw [redE]   (1) to (3);
    \draw [greenE] (2) to (3);
    \draw [greenE] (0) to (1);
    \draw [blueE]  (1) to (2);
    \draw [redE]   (0) to (2); 
  \end{scope}
\end{scope}
\end{tikzpicture}
\caption{The set $L_3$ of $3$-edge-coloured graphs that expresses line graphs of
$3$-partite graphs.}
\label{fig:3-colours}
\end{figure}

\begin{proposition}
\label{prop:linegraph-k-partite}
  For a graph $G$, the following statements are equivalent:
  \begin{enumerate}
    \item $G$ is the line graph of a $k$-colourable graph.
    \item $G$ admits an $L_k$-free $k$-edge-colouring.
    \item $G$ admits a $k$-edge-colouring with no monochromatic $P_3$, such that
      every vertex is incident with at most $2$ colours, and all triangles
      are either monochromatic or rainbow.
  \end{enumerate}
\end{proposition}

\begin{proof}
The last two items are clearly equivalent. We first prove that the first
statement implies the third one. Let $G$ be the line graph of a $k$-partite
graph $H(X_1,\dots, X_k)$. Notice that each edge $e$ of $G$ corresponds to a
pair of edges of $H$ incident in some $X_i$. This defines a $k$-edge-colouring
of $G$. It is straightforward to verify that this edge-colouring of $G$
satisfies the third statement. We see that the third item implies the first one.
Since there is no monochromatic $P_3$ in $G$, each colour class induces a
cluster. This induces a partition of the edge set of $G$ into maximal complete
subgraphs. Moreover, no vertex belongs to more than two of these subgraphs
because every vertex is incident with at most $2$-colours (and each colour class
is a cluster). Thus, $G$ is the line graph of some graph $H$ (Theorem 1
in~\cite{beinekeJCTB9}). Finally, it is not hard to see that the
$k$-edge-colouring of $G$ translates to a $k$-partition of $H$ into independent
sets. The claim follows.
\end{proof}

To conclude this brief section, we see that expression by forbidden coloured
graphs include certain edge-partition families. Given a pair of hereditary
properties $\mathcal{P}$ and $\mathcal{Q}$, a $(\mathcal{P,Q})$-edge-partition
of a graph $G$, is a partition $(B,R)$ of the edge set of $G$ such that $G-R$
belongs to $\mathcal{P}$, and $G-B$ belongs to $\mathcal{Q}$. Clearly, if
$\mathcal{P}$ and $\mathcal{Q}$ have a finite set of minimal obstructions, then
the class of graphs that admit a $(\mathcal{P,Q})$-edge-partition is expressible
by forbidden $2$-edge-coloured graphs. Indeed, for a graph $F$ let $B(F)$
(resp.\ $R(F)$) denote the set of $2$-edge-coloured graphs $F'$ such that the
blue subgraph (resp.\ red subgraph) of $F'$ is isomorphic to $F$. For instance,
$B(\overline{P_3})$ is the set $\{\B, \RB, \RRB\}$. It is evident that  the blue
subgraph of a $2$-edge-coloured graph $G'$ is $F$-free if and only if $G$ is
$B(F)$-free. Thus, by considering the union of $B(H)$ and $R(F)$, where $H$ and
$F$ range over  minimal obstructions of $\mathcal{P}$ and $\mathcal{Q}$
respectively, we obtain a set that expresses
$(\mathcal{P,Q})$-edge-partitionable graphs by $2$-edge-coloured graphs.

The observation above also holds for $k$-edge-coloured graphs. Let $k$ be a
positive integer and $i\le k$. For a graph $H$, we denote by $i(H,k)$ the set of
$k$-edge-coloured graphs $F'$ such that the $i$-colour subgraph of $F'$ is
isomorphic to $H$. The following lemma is self-evident.

\begin{lemma}
\label{i-H-k}
  Consider a pair of graphs $G$ and $H$, and let $i$ and $k$ be a pair of
  positive integers with $i \le k$. If $G'$ is a $k$-edge-colouring of $G$, then
  $G'$ is $i(H,k)$-free if and only if the $i$-colour subgraph of $G'$ is
  $H$-free.
\end{lemma}

To illustrate this lemma, we consider $2$-edge-colourings and the graph $P_3$.
Recall that a graph $G$ is $P_3$-free if and only if it is a cluster.  Since
$1(P_3,2) = \{\BB,\RBB\}$ and $2(P_3,2) = \{\RR, \RRB\}$, by \Cref{i-H-k} we
conclude that the edge set of graph $G$ can be covered by two clusters if and
only if it admits a $\{\RR, \BB, \RRB, \RBB\}$-free $2$-edge-colouring. We
already observed that a graph $G$ admits a $\{\RR, \BB, \RRB, \RBB\}$-free
$2$-edge-colouring if and only if it is a line graph of a bipartite graph
(\cref{prop:J3J4J8J9}). Hence, putting both these observations together, we
recover the well-known result stating that a graph is the line graph of a
bipartite graph if and only if its edge-set can be covered by two
clusters~\cite{petersonDAM126}. In general, using \Cref{i-H-k} we conclude the
following statement.

\begin{proposition}
\label{prop:edge-partition}
  Let $\mathcal{P}_1$, \dots, $\mathcal{P}_k$ be $k$ hereditary properties with
  set of minimal obstructions $M_i$ for $i\in\{1,\dots, k\}$. If
  $M_i$ is a finite set for each $i\in\{1,\dots, k\}$, then the class of graphs
  whose edge set admits a partition $(E_1,\dots, E_k)$ such that the spanning
  subgraph with edge set $E_i$ belongs to $\mathcal{P}_i$, is expressible by
  forbidden $k$-edge-coloured graphs.
\end{proposition}

In \cref{tab:edge-partition} we illustrate \cref{prop:edge-partition} by showing
some sets of $2$-edge-coloured graphs on three vertices together with the
edge-partition graph class that they express.

\begin{table}[ht!]
\begin{center}
  \begin{adjustbox}{max width=0.9\textwidth}
  \rowcolors{2}{white}{gray!25}  
  \begin{tabular}{| c | l |}
    \hline
      $\calF$
    & Edge-partition class expressed by $\calF$ \\
    \hline 
      $\{\B, \RB, \RRB\} \cup \{\R, \RB, \RBB\}$
    & (Complete multipartite, Complete multipartite) \\
    \hline
      $\{\B, \RB, \RRB\} \cup \{\RR, \RRB\}$
    & (Complete multipartite, Cluster) \\
    \hline
      $\{\B, \RB, \RRB\} \cup \{\RT\}$
    & (Complete multipartite, Triangle-free) \\
    \hline
      $\{\B, \RB, \RRB\} \cup \{\RR, \RRB, \RT\}$& (Complete multipartite,
      $\Delta \le 2$) \\
    \hline
      $\{\B, \RB, \RRB\} \cup \{\R, \RB, \RT,\RBB\}$
    & (Complete multipartite, Complete bipartite) \\
    \hline
      $ \{\BB, \RBB\} \cup  \{\RR, \RRB\}$
    & (Cluster, Cluster)  \\
    \hline
      $ \{\BB, \RBB\} \cup  \{\RT\}$
    & (Cluster, Triangle-free)  \\
    \hline
      $\{\BB, \RBB\} \cup \{\RR, \RRB, \RT\}$
    & (Cluster, $\Delta \le 2$) \\
    \hline
      $ \{\BB, \RBB\}   \cup \{\R, \RB, \RT,\RBB\}$
    &  (Cluster,  Complete bipartite) \\
    \hline
      $\{\BT\} \cup \{\RT\}$
    & (Triangle free, Triangle free) \\
    \hline
      $\{\BT\} \cup \{\RR, \RRB, \RT\}$
    & (Triangle free, $\Delta \le 2$) \\
    \hline
      $\{\BT\} \cup \{\R, \RB, \RT,\RBB\}$
    & (Triangle free, Complete bipartite) \\
    \hline
      $\{\BB, \RBB,\BT\} \cup \{\R, \RB, \RT,\RBB\}$
    & ($\Delta \le 2$, $\Delta \le 2$) \\
    \hline
      $\{\BB, \RBB,\BT\} \cup \{\R, \RB, \RT,\RBB\}$
    & ($\Delta \le 2$, Complete bipartite) \\
    \hline
  \end{tabular}
  \end{adjustbox}
  \caption{Edge-partition classes expressed by $2$-edge-coloured graphs on three
  vertices (see \cref{i-H-k,prop:edge-partition}). Notice that some of these
  edge-partition classes have simpler descriptions. For instance, a graph $G$
  admits an edge-partition into two spanning complete multipartite graphs if and
  only if $G$ is a complete multipartite graph. On the contrary, it is
  $NP$-complete to test if a graph $G$ admits an edge-partition into two
  triangle-free graphs \cite{burrMRT1990}.}
  \label{tab:edge-partition}
\end{center}
\end{table}

 
\section{Conclusions}
\label{sec:conclusions}

We studied classes expressible by sets of 2-edge-coloured graphs on three
vertices and focused on both structural and complexity aspects. Regarding
structural results, we focused mainly on the following two cases and managed to
relate the resulting families to other well-known hereditary classes:
\begin{enumerate}
    \item Graphs expressible by a set $\cal F$ consisting of 2-edge coloured
      graphs on three vertices and at most two edges. We succeeded in providing
      a characterization of all such classes as shown in \Cref{thm:two-edges}.
	\item Graphs expressible by a set $\cal F$ containing both monochromatic
      paths, and thus subclasses of perfect graphs. In this case we also
      succeeded in proposing a characterizations of all such classes presented
      in \Cref{fig:landscape}.
\end{enumerate}

For most of these characterizations we presented the list of minimal
obstructions to the class expressed by $\calF$. However, we left two such
characterizations open: list the minimal obstructions of skew-augmentations of
bipartite multigraphs, and of $X$-augmentations of line graphs of incidence
graphs (compare to \cref{RR-BB-RRB,BB-RR-BT-RRB}).

We further focused on the algorithmic perspective and showed, using tools from
constraint satisfaction theory, that several sets of $2$-edge-coloured graphs
express a class which can be recognized in polynomial time
(\Cref{tab:complexity-landscape}). We also show that whenever this method yields
a reduction to an $\NP$-complete CSP, and $\calF$ is a set of triangles, then
$\calF$ expresses a class which is $\NP$-hard to recognize
(\Cref{cor:triangles}). We underline that this is not necessarily always the
case: we use our structural results to show that some sets $\calF$ express a
class recognizable in polynomial-time (Lemma~\ref{lem:elementary-cases}), while
the CSP approach provided a reduction to an $\NP$-complete problem. We presented
two open questions regarding the computational complexity of the $\calF$-free
colouring problem (Questions~\ref{qst:complete-obstructions}
and~\ref{qst:NP-intermediate}).

\Cref{sec:general} also provides an obvious way and a guiding line to generalize
our results towards larger patterns and/or more colours. Since neither
structural nor complexity classifications for forbidden sets of
$2$-edge-coloured graphs on three vertices are complete, we propose the
following problem.

\begin{problem}
Classify the complexity of recognizing each graph class expressed by a set of
$2$-edge-coloured-graphs on three vertices. 
\end{problem}

\begin{problem}
List all graph classes expressed by a set of $2$-edge-coloured-graphs on three
vertices.
\end{problem}

As mentioned above, characterizations of graph classes by forbidden orientations
have been already considered in the literature~\cite{guzmanArXiv,skrienJGT6}. We
ask if the expressive power of forbidden orientations and of forbidden
$2$-edge-coloured graphs are incomparable, i.e., are there graph classes
expressible by forbidden orientations and not by forbidden $2$-edge-coloured
graphs? (and vice versa). In particular, we believe that there is no finite set
of $2$-edge-coloured graphs expressing bipartite graphs, and no finite set of
oriented graphs characterizing co-bipartite graphs.


\appendix


\section{Limitations of the ``natural'' reduction for sets on three vertices.}
\label{ap:complex}

For a positive integer $n$, an \textit{$n$-ary operation} on a (vertex) set $V$
is a function $f\colon V^n\to V$. Given a relational signature $\tau$ and a
$\tau$-structure $\mathbb A$, we say that $\mathbb A$ is \textit{preserved by an
operation} $f\colon V(\mathbb A)^n\to V(\mathbb A)$ if, for every $R \in \tau$
of arity $r$ and $(x^1_1,\dots x^1_r),\dots, (x^n_1,\dots, x^n_r)\in R(\mathbb
A)$, the tuple $(f(x^1_1,\dots, x^n_1), \dots, f(x^1_r,\dots, x^n_r))$ belongs
to $R(\mathbb A)$. For instance, if $G$ is a graph, then a binary operation
$f\colon V(G)^2$ $\to V(G)$ preserves $G$ if and only if $f$ is a homomorphism
from the product $G\times G$ to $G$.\footnote{In graph theoretic contexts, this
product is sometimes called the tensor or the weak product (see,
e.g.,~\cite[Exercise 14.1.18]{bondy2008})} This intuition extends to arbitrary
relational structures.   For a relational structure $\mathbb{A}$, the
\textit{$n$-th power} of $\mathbb A$, denoted $\mathbb A^n$, has vertex set
$V(\mathbb A)^n$ and for each $R \in \tau$ of arity $r$ a tuple $((x^1_1,\dots
x^1_n),\dots, (x^r_1,\dots, x^r_n))$ belongs to the interpretation of $R$ in
$\mathbb A^n$ if and only if $(x^1_1, \dots, x^r_1),\dots, (x^1_n,\dots,
x^r_n)\in R(\mathbb A)$.  An $n$-ary operation $f\colon V(\mathbb A)^n\to
V(\mathbb A)$ \textit{preserves} $\mathbb A$ if and only if $f$ is a
homomorphism from $\mathbb A^n$ to $\mathbb A$.

\begin{theorem}[Schaefer's theorem]\label{thm:schaefer}
    For every structure $\mathbb B$ with a two-element vertex set, either
    $\CSP(\mathbb B)$ is NP-complete, or $\mathbb B$ is preserved by one of the
    following operations,
    \begin{itemize}
        \item the binary \emph{minimum or maximum operation}, i.e., an operation
            $f$ satisfying $f(x,y)=f(y,x)$ and $f(x,x) = x$ for all $x,y \in
            V(\mathbb B)$, 
        \item the \emph{ternary majority operation}, i.e., the (unique)
            operation $f$ satisfying $$f(x,x,y) = f(x,y,x) = f(y,x,x) = x$$ for
            all $x,y \in V(\mathbb B)$, 
        \item the \emph{ternary minority operation}, i.e., the (unique)
            operation $f$ satisfying $$f(x,x,y) = f(x,y,x) = f(y,x,x) = y$$ for
            all $x,y \in V(\mathbb B)$, 
        \item a constant operation, i.e., an operation satisfying $f(x) = f(y)$
            for all $x,y \in V(\mathbb B)$. 
    \end{itemize}
    In all of these cases, $\CSP(\mathbb B)$ is polynomial-time solvable. 
\end{theorem}
Note that for some sets $\calF$, the reduction given by \Cref{lem:boolean-CSP}
is simply a reduction to an NP-complete problem, and for others, we reduce to a
problem known to be solvable in polynomial time. One can distinguish between
these cases using Schaefer's theorem and verifying whether the boolean structure
$\mathbb B_\calF$ is preserved by one of the operations listed
in~\Cref{thm:schaefer}. In particular, the case when $\mathbb B_\calF$ is
preserved by a constant operation is straightforward to phrase in terms of the
forbidden $2$-edge-coloured graphs in $\calF$.  Recall that a set $\calF$ is
trivial when there is a colour $c$ such that for every graph $G$, either $\calF$
contains all    $2$-edge-colourings of $G$, or $\calF$ does not contain the
$c$-monochromatic colouring of $G$.

\begin{lemma}\label{lem:constant}
    Let $\calF$ be a finite set of $2$-edge-coloured graphs. Then $\mathbb
    B_\calF$ is preserved by a constant operation if and only if $\calF$ is a
    trivial set.
\end{lemma}

In general, there does not need to be a simple description of those finite sets
$\calF$ of $2$-edge-coloured graphs such that $\mathbb B_\calF$ is preserved by
the minimum, the maximum, the minority, or the majority operations. However, it
is a decidable process and in particular, we can list such sets when restricted
to $2$-edge-coloured graphs on three vertices.

The reader who is not familiar with operations on a boolean domain can think of
the maximum operation acting on $2$-edge-colourings $c_1,c_2\colon E(G)\to
\{r,b\}$ where $\max(c_1,c_2)(e) = b$ if $c_1(e) =b $ or $c_2(e) = b$, and
$\max(c_1,c_2)(e) = r$ otherwise --- we choose the convention that $r$ is the
minimum and $b$ is the maximum from $\{r,b\}$. Similarly, given three colourings
$c_1,c_2,c_3\colon E(G)\to \{r,b\}$ the minority colouring
$\minority(c_1,c_2,c_3)\colon E(G)\to \{r,b\}$ is defined by $c_1(e)$ if the
three colourings agree on $e$, and otherwise $c_i(e)$ where $c_j(e)\neq c_i(e)$
for $j\in\{1,2,3\}\setminus\{i\}$; in simple words, upon disagreement of
colours, the minority colouring takes the minority vote. Majority and minimum
operation on colourings can be defined analogously. We say that a set of
$2$-edge-colourings $C$ of a graph $G$ is preserved by the maximum operation if
for any $c_1,c_2\in C$, it is the case that $\max(c_1,c_2)\in C$. Similarly, we
say that $C$ is preserved by the maximum operation, the majority operation, and
the minority operation. The following relation between operations preserving
$\calF$-free $2$-edge-colourings of graphs, and operations preserving $\mathbb
B_\calF$ is straightforward to observe.

\begin{observation}\label{obs:preserve-colourings}
  Let $\calF$ be a set of $2$-edge-coloured graphs, and $\mathcal U$ the set
  of underlying graphs of $\calF$. The following statements are equivalent
  for any operation $f\in\{\min, \max, \minority, \majority\}$.
  \begin{itemize}
    \item The structure $\mathbb B_\calF$ is preserved by $f$.
    \item For every graph $G$ the set of $\calF$-free colourings of $G$ is
      preserved by $f$. 
    \item For every graph $G\in \mathcal U$ the set of $\calF$-free colourings
      of $G$ is preserved by $f$.
  \end{itemize}
\end{observation}

Using this observation, it is (tedious, but) straightforward to compute the sets
$\calF$ of $2$-edge-coloured graphs on three vertices for which $\mathbb
B_\calF$ is preserved by the minimum, maximum, majority, or minority operation.
In particular, for any set $\calF$ of $2$-edge-colourings of $\overline{P_3}$,
it is easy to observe that $\mathbb B_\calF$ is preserved by all operations from
Schaefer's theorem. Similarly, if $\calF$ is the set of all $2$-edge-colourings
of some graph $G$, then $\mathbb B_\calF$ is also preserved by all these
operations. In Table~\ref{tab:preserve-P3} we list the proper non-empty subsets
$\calF\subseteq \{\RR,\BB, \RB\}$, and the operations that preserve $\mathbb
B_\calF$ in each case. Similarly, in Table~\ref{tab:preserve-K3} we list the
sets of $2$-edge-coloured triangles $\calF$ together with the operations that
preserve $\mathbb B_\calF$. 

\begin{remark}\label{rmk:colour-symmetry} If $\mathbb B_\calF$ is preserved by
maximum (resp.\ minimum), and $\calF'$ is obtained by changing blue edges by
red edges, and vice-versa, then $\mathbb B_{\calF'}$ is preserved by minimum
(resp.\ maximum). Moreover,  $\mathbb B_\calF$ is preserved by the majority
(resp.\ minority) operation  if and only if $\mathbb B_{\calF'}$  is preserved
by majority (resp.\ minority) operation.
\end{remark}

\begin{table}[ht!]
\begin{center}
  \begin{adjustbox}{max width=0.9\textwidth}
  \rowcolors{2}{white}{gray!25}  
  \begin{tabular}{| c | c | c |}
    \hline
     $\calF$
    & $\calF$-free paths
    & Operations preserving $\mathbb B_\calF$ \\
    \hline
         $\{\RR\}$
      & $\{\BB, \RB\}$
      &  maximum, majority\\
    \hline
      $\{\BB\}$
      & $\{\RR, \RB\}$
      &  minimum, majority\\
    \hline
      $\{\RB\}$
      & $\{\RR, \BB\}$
      &  all\\
  \hline
  $\{\RR, \BB\}$
      & $\{\RB\}$
      & majority, minority \\
  \hline
  $\{\RR, \RB\}$
      & $\{\BB\}$
      &  all\\
  \hline
  $\{\BB, \RB\}$
      & $\{\RR\}$
      &  all\\
  \hline
  $\{\RR, \BB, \RB\}$
      & $\varnothing$
      &  all\\
  \hline
  \end{tabular}
  \end{adjustbox}
  \caption{Non-empty sets of $2$-edge-coloured paths on three vertices $\calF$
  together with the $\calF$-free $2$-edge-coloured paths on three vertices, and
  the non-constant operations from Theorem~\ref{thm:schaefer} that preserve
  $\mathbb B_\calF$. }
  \label{tab:preserve-P3}
\end{center}
\end{table}

\begin{table}[ht!]
\begin{center}
  \begin{adjustbox}{max width=0.9\textwidth}
  \rowcolors{2}{white}{gray!25}  
  \begin{tabular}{| c | c | c |}
    \hline
     $\calF$
    & $\calF$-free triangles
    & Operations preserving $\mathbb B_\calF$ \\
    \hline
         $\{\RT\}$
      & $\{\BT, \RRB, \RBB\}$
      & maximum \\
    \hline
      $\{\RRB\}$
      & $\{\RT, \BT, \RBB\}$
      &  maximum\\
  \hline
  $\{\RT,\BT\}$
      & $\{\RRB, \RBB\}$
      &  none\\
  \hline
  $\{\RT,\RRB\}$
      & $\{\BT, \RBB\}$
      &  maximum, majority\\
  \hline
  $\{\RT,\RBB\}$
      & $\{\BT, \RRB\}$
      &  minority \\
  \hline
  $\{\RRB, \RBB\}$
      & $\{\RT,\BT\}$
      &  all\\
  \hline
  $\{\RT, \BT, \RRB\}$
      & $\{\RBB\}$
      & none\\
  \hline
  $\{\RT,  \RRB\, \RBB\}$
      & $\{\BT\}$
      &  all \\
  \hline
  $\{\RT, \BT, \RRB\, \RBB\}$
      & $\varnothing$
      &  all \\
  \hline
  
  \end{tabular}
  \end{adjustbox}
  \caption{Up to colour symmetry (see Remark~\ref{rmk:colour-symmetry}), all
  non-empty set of $2$-edge-coloured triangles $\calF$ together with the
  $2$-edge-free coloured triangles, and the non-constant operations from
  Theorem~\ref{thm:schaefer} that preserve $\mathbb B_\calF$.}
  \label{tab:preserve-K3}
\end{center}
\end{table}

From Observation~\ref{obs:preserve-colourings} together
with~\Cref{tab:preserve-P3,tab:preserve-K3} (and
Remark~\ref{rmk:colour-symmetry}), one can read off which reductions from
Lemma~\ref{lem:boolean-CSP} yield a polynomial-time algorithm, and for which
$\CSP(\mathbb B_\calF)$ is NP-complete

\end{document}